\documentclass[10pt,reqno]{amsart}

\usepackage[all]{xy}
\usepackage{tom}
\usepackage{epic,eepic}

\newcommand{\ux}{{\boldsymbol x}}
\newcommand{\uz}{{\boldsymbol z}}
\newcommand{\Kx}{\K[[\ux]]}
\newcommand{\qh}{QH}

\newcommand{\rsqh}{rSQH}

\newcommand{\ND}{ND}
\newcommand{\SND}{IND}
\newcommand{\NND}{NND}
\newcommand{\WND}{WND}
\newcommand{\SNND}{INND}
\newcommand{\WNND}{WNND}

\DeclareMathOperator{\In}{in}
\DeclareMathOperator{\gr}{gr}
\DeclareMathOperator{\jj}{j}
\DeclareMathOperator{\tj}{tj}

\newcommand{\tom}[1]{}
\newcommand{\lang}[1]{}
\newcommand{\laenger}[1]{}
\newcommand{\kurz}[1]{#1}
\newcommand{\bmath}{\kurz{\begin{math}}\lang{\begin{displaymath}} }
\newcommand{\emath}{\kurz{\end{math}}\lang{\end{displaymath}} }

\kurz{\renewcommand{\K}{K}
  \newcommand{\LL}{L}
  \newcommand{\ideal}{\subset}
  \newcommand{\unideal}{\subseteq}}
\lang{\renewcommand{\K}{{\mathds K}}
  \newcommand{\LL}{{\mathds L}}
  \newcommand{\ideal}{\lhd}
  \newcommand{\unideal}{\unlhd}
  }

\catcode`@=11
\renewenvironment{enumerate}{%
  \listparindent0pt 
  \ifnum \@enumdepth >3 \@toodeep\else
      \advance\@enumdepth \@ne
      \edef\@enumctr{enum\romannumeral\the\@enumdepth}\list
      {\csname label\@enumctr\endcsname}{\usecounter
        {\@enumctr}\def\makelabel##1{\hss\llap{\upshape##1}}}\fi
      \itemsep1ex\partopsep1ex\labelsep1ex
}{%
  \endlist
}
\catcode`@=\active%

\begin{document}

   \parindent0cm

   \title{Invariants of Hypersurface Singularities in Positive Characteristic}
   \author{Yousra Boubakri}
   \address{Universit\"at Kaiserslautern\\
     Fachbereich Mathematik\\
     Erwin--Schr\"odinger--Stra\ss e\\
     D --- 67663 Kaiserslautern
     }
   \email{yousra@mathematik.uni-kl.de}
   \author{Gert-Martin Greuel}
   \address{Universit\"at Kaiserslautern\\
     Fachbereich Mathematik\\
     Erwin--Schr\"odinger--Stra\ss e\\
     D --- 67663 Kaiserslautern\\
     Tel. +496312052850\\
     Fax +496312054795
     }
   \email{greuel@mathematik.uni-kl.de}
   \urladdr{http://www.mathematik.uni-kl.de/\textasciitilde greuel}
   \author{Thomas Markwig}
   \address{Universit\"at Kaiserslautern\\
     Fachbereich Mathematik\\
     Erwin--Schr\"odinger--Stra\ss e\\
     D --- 67663 Kaiserslautern\\
     Tel. +496312052732\\
     Fax +496312054795
     }
   \email{keilen@mathematik.uni-kl.de}
   \urladdr{http://www.mathematik.uni-kl.de/\textasciitilde keilen}



   \date{\today}
   
     
   \begin{abstract}
     We study singularities $f\in\K[[x_1,\ldots,x_n]]$ over an
     algebraically closed field $\K$ of arbitrary characteristic with
     respect to right respectively contact equivalence, and we establish that the
     finiteness of the Milnor respectively the Tjurina number is
     equivalent to finite determinacy. We  give improved bounds for
     the degree of determinacy in positive characteristic. Moreover,
     we consider different non-degeneracy conditions of Kouchnirenko,
     Wall and Beelen-Pellikaan in positive characteristic, and we show
     that planar Newton non-degenerate singularities 
     satisfy Milnor's formula $\mu=2\cdot\delta-r+1$. This implies the
     absence of \emph{wild} vanishing cycles in the sense of Deligne. 
   \end{abstract}

   \maketitle


   \section{Introduction}

   \lang{
     Throughout this paper $\K$ shall be an algebraically closed field of
     arbitrary characteristic unless explicitly stated otherwise. By
     \begin{displaymath}
       \Kx=\K[[x_1,\ldots,x_n]]= 
       \left\{\sum_{\alpha\in\N^n}a_\alpha\cdot \ux^\alpha\;\Big|\;a_\alpha\in\K\right\}
     \end{displaymath}
     we denote the formal power series ring over $\K$ in $n\geq 2$
     indeterminates $x_1,\ldots,x_n$ using the usual multiindex notation
     $\ux^\alpha=x_1^{\alpha_1}\cdots x_n^{\alpha_n}$ for
     $\alpha=(\alpha_1,\ldots,\alpha_n)\in\N^n$. Moreover, we denote by
     \begin{displaymath}
       \m=\langle x_1,\ldots,x_n\rangle\ideal \Kx
     \end{displaymath}
     the unique maximal ideal of $\Kx$, so that the set of units in
     $\Kx$ is $\Kx^*=\Kx\setminus\m$.
   }
   \kurz{
     Throughout this paper $\K$ denotes an algebraically closed field of
     arbitrary characteristic unless explicitly stated otherwise,
     $\Kx=\K[[x_1,\ldots,x_n]]$
     denotes the formal power series ring over $\K$ and 
     $\m=\langle x_1,\ldots,x_n\rangle$ 
     the unique maximal ideal of $\Kx$.
   }

   \lang{
     If we are interested in
     the geometry of a power series $f\in\Kx$ there are two natural
     equivalence relations with respect to which we may want to
     consider it. 
   }

   We say that two power series $f,g\in\Kx$ are
   \emph{right equivalent} \lang{to each other} if \lang{and only if} there is an
   automorphism $\varphi\in\Aut(\Kx)$ such that $f=\varphi(g)$, and we
   denote this by $f\sim_r g$. 
   \lang{
     If we replaced $\K$ by the complex
     numbers and formal power series by convergent ones then $\varphi$
     would induce an isomorphism of the zero fiber of $f$ as well as of
     close by fibers. That is how we should interpret 
     right equivalence also in this more general setting.

     If we are only interested in the geometry of the zero fiber, then
     the second equivalence relation is the appropriate one. }
   We call $f,g\in\Kx$ 
   \emph{contact equivalent} \lang{to each other} if \lang{and only if} there is an
   automorphism $\varphi\in\Aut(\Kx)$ and a unit $u\in\Kx^*$ such that
   $f=u\cdot \varphi(g)$, and we
   denote this by $f\sim_c g$. 
   \lang{The idea here is, that $\varphi$ and
     $u$ still induce an isomorphism of the zero fibers of $f$ and $g$. }

   \lang{However, we have to replace the geometric notion of the zero fiber
   by the algebraic counterpart of its coordinate ring. That is, for a
   power series }\kurz{For } $f\in\Kx$ we call the analytic $\K$-algebra
   $R_f=\Kx/\langle f\rangle$ the 
   induced \emph{hypersurface singularity}. 
   \kurz{Obviously, two power series are contact equivalent if and
     only if the induced hypersurface singularities are isomorphic as
     local $\K$-algebras.} \lang{We obviously have
     \begin{displaymath}
       f\sim_c g\;\;\;\Longleftrightarrow\;\;\; R_f\cong R_g,
     \end{displaymath}
     i.e.\ $f$ and $g$ are contact equivalent if and only if the
     induced hypersurface singularities are isomorphic as local
     $\K$-algebras.} 

   Right equivalence as well as contact equivalence can be expressed
   via group actions. The group $\mathcal{R}=\Aut(\Kx)$ operates on $\Kx$ and the
   equivalence classes of $\Kx$ with respect to right equivalence are
   the orbits of this group, which we therefore call the \emph{right
     group}.  Similarly, the semidirect product
   \bmath
     \mathcal{K}=\Kx^*\ltimes\Aut(\Kx)
   \emath
   with multiplication
   \bmath
     (u,\varphi)\cdot(v,\psi)=(u\cdot\varphi(v),\varphi\circ\psi)
   \emath
   operates on $\Kx$ and the orbits of this group operation are the
   equivalence classes with respect to
   contact equivalence. $\mathcal{K}$ is also known as the \emph{contact
     group}.

   Over the complex numbers we say that the origin is an
   isolated singular point of $f$ if $f$ is not singular at any point
   close-by\lang{, i.e.\ the origin is the only common zero of the partial
     derivatives of $f$}. We reformulate this algebraically so 
   that it works over any field. For a power series $f\in\Kx$ we
   denote by
   \bmath
     \jj(f)=\langle f_{x_1},\ldots,f_{x_n}\rangle \unideal  \Kx
   \emath
   the \emph{Jacobian ideal} of $f$, \lang{i.e.\ the ideal generated by the
   partial derivatives of $f$,} and we call the associated algebra
   \bmath
     M_f=\Kx/\jj(f)
   \emath
   the \emph{Milnor algebra} of $f$ and its dimension
   \bmath
     \mu(f)=\dim_\K(M_f)
   \emath
   the \emph{Milnor number} of $f$. We then call \lang{the origin an
   isolated singular point of $f$, or we simply call } $f$ an
   \emph{isolated singularity} if $\mu(f)<\infty$, which is 
   equivalent to the existence of a positive integer $k$ such that
   $\m^k\subseteq\jj(f)$. 

   \lang{Similarly, over $\C$ we would call the origin an isolated singular
   point of the hypersurface singularity defined by $f$ if this
   hypersurface singularity has no other singular point close-by,
   i.e.\ the origin is the only common zero of $f$ and its partial
   derivatives. Algebraically we thus}\kurz{Now} consider the \emph{Tjurina ideal}
   \bmath
     \tj(f)=\langle f,f_{x_1},\ldots,f_{x_n}\rangle=\langle f\rangle+\jj(f)\unideal\Kx
   \emath
   of $f$, the associated \emph{Tjurina algebra}
   \bmath
     T_f=\Kx/\tj(f)
   \emath
   of $f$ and its dimension
   \bmath
     \tau(f)=\dim_\K(T_f),
   \emath
   the \emph{Tjurina number} of $f$. We then call \lang{the origin an
   isolated singular point of the hypersurface singularity $R_f$, or
   we will simply call } $R_f$ an \emph{isolated hypersurface singularity} if
   $\tau(f)<\infty$ or, equivalently, if there is a positive integer
   such that $\m^k\subseteq \tj(f)$. \kurz{Over the complex numbers
     this is equivalent to say that the zero-set of $f$ is not
     singular at any point close to the origin.}

   It is straight forward to see
   \lang{(\cite[Lem.~1.2.7]{Bou09}) } that for an automorphism $\varphi\in\Aut(\Kx)$ and a unit
   $u\in\Kx^*$ we have
   \bmath
     \jj\big(\varphi(f)\big)=\varphi\big(\jj(f)\big)
   \emath
   and
   \bmath
     \tj\big(u\varphi(f)\big)=\varphi(\tj(f)\big).
   \emath
   In particular, the Milnor number is \emph{invariant} under right equivalence and
   the Tjurina number is \emph{invariant} under contact equivalence. 
   
   It is a non-trivial theorem, using methods from complex analysis, which
   cannot be extended to other fields that the
   Milnor number is indeed invariant under contact
   equivalence (see \cite[p.~262]{Gre75}), and it is even a topological
   invariant (see \cite{LR76}). Using the 
   Lefschetz Principle the result for contact equivalence can be 
   generalised to arbitrary algebraically closed fields of
   characteristic zero (see \cite[Prop.~5.2.1,Prop.~5.3.1]{Bou09} for
   a detailed proof of Theorem \ref{thm:lefschetz1} and Theorem
   \ref{thm:lefschetz2}).  

   \begin{theorem}\label{thm:lefschetz1}
     Let $\K$ be an algebraically closed field of characteristic zero
     and $f,g\in\Kx$. 

     If $f\;\sim_c\;g$, then $\mu(f)=\mu(g)$.
   \end{theorem}
   \begin{proof}
     Since the Milnor number is invariant under right equivalence, it
     suffices to show that $\mu(f)=\mu(u\cdot f)$ for any unit
     $u\in\Kx^*$. If $A$ denotes the subset of $\K$
     containing the coefficients of $u$, $f$ and all partial
     derivatives $f_{x_i}$ of $f$, then $A$ is at most countable
     infinite. Since $\Char(\K)=0$ and since $\Q\subset\C$ is a field
     extension of uncountable transcendence degree 
     the field $\Q(A)$ is isomorphic to a subfield $\LL$
     of the field $\C$ of complex numbers, and we may suppose that $f$
     and $u\cdot f$ belong to $\LL[[\ux]]\subseteq\C[[\ux]]$. Now using
     the fact that over the complex numbers $f$ and $u\cdot f$ have
     the same Milnor number we get
     \begin{align*}
       \mu(f)=&\dim_\LL(\LL[[\ux]]/\jj(f))=\dim_\C(\C[[\ux]]/\jj(f))
       \\=&\dim_\C(\C[[\ux]]/\jj(u\cdot f))=\dim_\LL(\LL[[\ux]]/\jj(u\cdot f))=\mu(u\cdot f).
     \end{align*}
   \end{proof}

   In positive characteristic this result does not hold any more;
   e.g.\ if $\Char(\K)=p>0$ and $f=x^p+y^{p-1}$, then $\mu(f)=\infty$
   while the contact equivalent series $g=(1+x)\cdot f$ has Milnor
   number $\mu(g)=p\cdot (p-2)$.

   A well-known result in complex singularity theory states that 
   the Milnor number of a power series is finite if and only if the Tjurina
   number is so\lang{, i.e.\ that the origin is an isolated singularity of
   $f$ if and only if it is an isolated singular point of the
   corresponding hypersurface} (see \cite[Lem.~2.3]{GLS07}). This fact can also be
   generalised to arbitrary fields of \emph{characteristic zero} using
   the Lefschetz principle.

   \begin{theorem}\label{thm:lefschetz2}
     Let $\K$ be an algebraically closed field of characteristic zero
     and $f\in\Kx$.

     Then $\mu(f)<\infty$ if and only if $\tau(f)<\infty$.    
   \end{theorem}
   \begin{proof}
     Let $A$ be the set of
     coefficients of $f$ and all its partial derivatives. As in the
     proof of Theorem \ref{thm:lefschetz1} the field $\Q(A)$ is
     isomorphic to a subfield $\LL$ of $\C$. We may therefore assume
     that $f\in\LL[[\ux]]\subset\C[[\ux]]$, so that
     \begin{displaymath}
       \mu(f)=\dim_\LL(\LL[[\ux]]/\jj(f))=\dim_\C(\C[[\ux]]/\jj(f)),
     \end{displaymath}
     \begin{displaymath}
       \tau(f)=\dim_\LL(\LL[[\ux]]/\tj(f))=\dim_\C(\C[[\ux]]/\tj(f)).
     \end{displaymath}
     Using the result for $\C$,  $\tau(f)$ is finite if and only if $\mu(f)$ is finite.     
   \end{proof}

   For fields of positive characteristic this is false. The
   same example as above shows that
   $\tau(f)=p\cdot (p-2)$ while $\mu(f)=\infty$. 

   \bigskip

   Our principle interest is the classification of  power series with
   respect to right respectively 
   contact equivalence\lang{, where the latter is the same as to say that we
   are interested in classifying hypersurface singularities up to
   isomorphism} (see \cite{BGM10a}). In order to do this we need
   finiteness conditions and therefore we
   restrict to the isolated case, i.e.\ to the case that $f$ is an
   isolated singularity for right equivalence and to the
   case that $R_f$ is an isolated hypersurface singularity  for
   contact equivalence, which are two distinct conditions
   in positive characteristic.

   A first important step in the attempt to classify singularities
   from a theoretical point of view as well as from a practical one is
   to know that the equivalence class is determined by a finite number
   of terms of the power series $f$ and to find the (smallest) corresponding degree
   bound.  We say that $f$ is \emph{right $k$-determined} if $f$ is
   right equivalent to every $g\in\Kx$ whose $k$-jet coincides with that of
   $f$, where the $k$-jet of $f$ is
   \bmath
     \jet_k(f)\lang{=\overline{f}}\in\Kx/\m^{k+1}
   \emath
   the residue class of $f$ modulo the $k+1$-st power of the maximal
   ideal. \lang{I.e.\ $f$ is right $k$-determined if it is right equivalent
   to every $g$ which coincides with $f$ up to order $k$.} Similarly,
   we call $f$ \emph{contact $k$-determined} if $f$ is
   contact equivalent to every $g$ whose $k$-jet coincides with that
   of $f$. In both situations we say that $f$ is \emph{finitely
     determined} if it is $k$-determined for some positive integer
   $k$, and we call the least such $k$ the \emph{determinacy} of $f$. 

   Over the complex numbers it is well known that  $f$ is finitely
   determined w.r.t.\ right or contact equivalence if and only if $f$
   \lang{respectively $R_f$} is an isolated singularity. It is
   straight forward to generalise this to any field of characteristic
   zero, using the infinitesimal characterisation of local
   triviality. Since the proof involves the solution of a
   differential equation, it does not work in positive
   characteristic. 
   \kurz{One main result of this paper is the following
     generalisation to arbitrary characteristic:
     Finite right
     respectively contact determinacy of $f$ is equivalent to the isolatedness
     of the singularity $f$ respectively $R_f$ (see Theorem
     \ref{thm:fd=isolated}).}
   \lang{We will, however, prove the following theorem.

     \begin{varthm-italic-break}[Theorem~\ref{thm:fd=isolated}]
       Let $0\not=f\in\m\ideal\Kx$ be a power series.
       \begin{enumerate}
       \item $f$ is an isolated singularity if and only
         if $f$ is finitely right determined.
       \item $R_f$ is an isolated hypersurface singularity if and only
         if $f$ is finitely contact determined.
       \end{enumerate}     
     \end{varthm-italic-break}
   }

   In the complex case and thus for arbitrary fields of characteristic
   zero it is known that the right determinacy is at most
   $\mu(f)+1$ and the contact determinacy is at most
   $\tau(f)+1$. For arbitrary characteristic it was shown in
   \cite{GK90} (among others) that $2\cdot \mu(f)$ respectively $2\cdot\tau(f)$ are
   bounds for the degree of right respectively 
   contact determinacy. We will improve these bounds in
   Section~\ref{sec:finitedeterminacy} substantially\lang{.}\kurz{ (see
     Theorem~\ref{thm:finitedeterminacybound} and
     Corollary~\ref{cor:finitedeterminacybound}).}  
   \lang{For the
     formulation of our result we introduce the \emph{order} $\ord(f)$ of a
     non-zero power series as the largest integer $k$ such that
     $f\in\m^k$, \lang{i.e.\ the smallest integer $k$ such that $f$ has non-zero
       terms of degree $k$,} and we set $\ord(0)=\infty$.
     $\ord(f)$ is invariant under right and contact equivalence.
     
     \begin{varthm-italic-break}[Theorem~\ref{thm:finitedeterminacybound}]
       Let $0\not=f\in\m^2\unideal \Kx$ and $k\in\N$.
       \begin{enumerate}
       \item If $\m^{k+2}\subseteq\m^2\cdot\jj(f)$, then $f$ is
         right $(2k-\ord(f)+2)$-determined. 
       \item If $\m^{k+2}\subseteq\m\cdot\langle f\rangle+\m^2\cdot\jj(f)$, then $f$ is
         contact $(2k-\ord(f)+2)$-determined. 
       \end{enumerate}
     \end{varthm-italic-break}
     
     \begin{varthm-italic-break}[Corollary~\ref{cor:finitedeterminacybound}]
       Let $0\not=f\in\m^2\unideal \Kx$.
       \begin{enumerate}
       \item If $\mu(f)<\infty$, then the right determinacy of $f$ is at
         most $2\mu(f)-\ord(f)+2$.
       \item If $\tau(f)<\infty$, then the contact determinacy of $f$ is
         at most $2\tau(f)-\ord(f)+2$. 
       \end{enumerate}       
     \end{varthm-italic-break}
   }

   The Milnor number of a singularity is governed by the geometry of
   its Newton diagram. In \cite{Kou76} Kouchnirenko introduced the
   \emph{Newton number} of a singularity which only depends on the
   Newton diagram, and he showed that this number is a lower bound
   for the Milnor number. Moreover, these two numbers
   coincide for non-degenerate and convenient singularities with fixed Newton
   diagram -- no matter what the characteristic of the base field 
   is. Unfortunately, his non-degeneracy assumption (\emph{Newton 
     non-degeneracy}, \NND) does not include all right semi-quasihomogeneous
   singularities (see page \pageref{page:rsqh}), which led Wall (see \cite{Wal99}) to the modified
   notion of \emph{inner Newton non-degeneracy} \SNND\ in characteristic
   zero.\footnote{Wall's notation is NPND$^*$, but we prefer the
     notation \emph{inner Newton non-degeneracy} since it is a
     condition only on the inner faces of the Newton diagram.}
   Neither of these notions fully implies the other, but there
   are well-understood relations, and Wall's notion allows to
   determine the Milnor number of a singularity from the geometry of
   the Newton diagram in the same way. In Section
   \ref{sec:nd} we  introduce the different notions of
   non-degeneracy and the Newton number $\mu_N(f)$, recall
   Kouchnirenko's result 
   and generalise Wall's result to positive characteristic\lang{.}
   \kurz{ (see Theorem~\ref{thm:snnd}).}
   \lang{

     \begin{varthm-italic-break}[Theorem~\ref{thm:kouchnirenko} \rm (Kouchnirenko, \cite{Kou76})]
       For $f\in\Kx$ we have $\mu_N(f)\leq \mu(f)$, and if $f$ is
       \NND\ and convenient then  $$\mu(f)=\mu_N(f)<\infty.$$
     \end{varthm-italic-break}
     
     \begin{varthm-italic-break}[Theorem~\ref{thm:snnd}
       \rm (Wall, \cite{Wal99})]
       If $f\in\Kx$ is \SNND\ w.r.t.\ some $C$-polytope, then
       $$\mu(f)=\mu_N(f)=\mu_N\big(\Gamma_-(f)\big)<\infty.$$
     \end{varthm-italic-break}

   }
   Moreover, we show that if the Newton diagram has only one
   facet then the right semi-quasihomogeneous singularities are
   precisely the inner Newton
   non-degenerate ones (see Proposition \ref{cor:snnd-rsqh}). 

   In the
   case of plane curve singularities we shall see that the condition of
   convenience can be dropped (in any
   characteristic) for the 
   equality of the Milnor and the Newton number (see Proposition
   \ref{prop:planarnnd}), and we show that
   Newton non-degeneracy implies inner Newton
   non-degeneracy (see Proposition \ref{prop:nnd-snnd}).

   Under the 
   assumption of \emph{weak Newton non-degeneracy}  \WNND, introduced
   by Beelen and Pellikaan, they 
   showed in \cite{BP00} how the delta invariant of a plane curve
   singularity can be computed in terms of the Newton
   diagram. We extend this to the non-convenient case and show that
   also in positive characteristic Newton non-degenerate singularities 
   satisfy Milnor's well-known formula
   $\mu(f)=2\cdot\delta(f)-r(f)+1$, relating the Milnor number and 
   the delta invariant\lang{.}\kurz{ (see Theorem~\ref{thm:milnorformula}).} 
   \lang{

     \begin{varthm-italic-break}[Theorem~\ref{thm:milnorformula}]
       If $f\in\Kx$ is \ND\ along each face of $\Gamma(f)$, then
       $\mu(f)=2\cdot\delta(f)-r(f)+1$. 
     \end{varthm-italic-break}

   }
   Using a result of Melle and Wall, this implies that $f$ has no
   \emph{wild} vanishing cycles in the sense of Deligne (see
   Corollary~\ref{cor:mellewall}). 

   \section{Finite determinacy}\label{sec:finitedeterminacy}

   In this section we prove that finite determinacy is
   equivalent to the isolatedness of the singularity. We start by
   showing that an isolated singularity is finitely determined and
   improve previously known determinacy bounds.

   \kurz{For the
     formulation of our result we introduce the \emph{order} $\ord(f)$ of a
     non-zero power series as the largest integer $k$ such that
     $f\in\m^k$, \lang{i.e.\ the smallest integer $k$ such that $f$ has non-zero
       terms of degree $k$,} and we set $\ord(0)=\infty$. 
     $\ord(f)$ is invariant under right and contact equivalence.
   }

   \begin{theorem}\label{thm:finitedeterminacybound}
     Let $0\not=f\in\m^2$ and $k\in\N$.
     \begin{enumerate}
     \item If $\m^{k+2}\subseteq\m^2\cdot\jj(f)$, then $f$ is
       right $(2k-\ord(f)+2)$-determined. 
     \item If $\m^{k+2}\subseteq\m\cdot\langle f\rangle+\m^2\cdot \jj(f)$, then $f$ is
       contact $(2k-\ord(f)+2)$-determined. 
     \end{enumerate}     
   \end{theorem}
   \begin{proof}
     We first consider the case of contact determinacy and set
     $o=\ord(f)$. It follows that $\ord(f_{x_i})\geq o-1$ for all
     $i=1,\ldots,n$ and by assumption we thus have
     \begin{displaymath}
       \m^{k+2}\subseteq\m\cdot\langle f\rangle+\m^2\cdot\langle
       f_{x_1},\ldots,f_{x_n}\rangle\subseteq\m^{o+1}. 
     \end{displaymath}
     This implies $k\geq o-1$.
     We set
     \begin{displaymath}
       N=2k-o+2\geq k+1,  
     \end{displaymath}
     and we consider $g\in\Kx$ such that
     $g-f\in\m^{N+1}$, i.e.\ $f$ and $g$ have the same $N$-jet. We
     have to show that $f$ and $g$ are contact equivalent, i.e.\ that
     there exists a unit $u\in\Kx^*$ and an automorphism
     $\varphi\in\Aut(\Kx)$ such that
     \begin{displaymath}
       g=u\cdot \varphi(f).
     \end{displaymath}
     We construct $u$ and $\varphi$ inductively, i.e.\ we
     construct inductively sequences of units $(u_p)_{p\geq 1}$ and
     of automorphisms $(\varphi_p)_{p\geq 1}$ such that
     $u_p\cdot\varphi_p(f)$ converges in the $\m$-adic topology to $u\cdot\varphi(f)$ for some
     unit $u\in\Kx^*$ and some automorphism $\varphi\in\Aut(\Kx)$ and
     at the same time
     \begin{displaymath}
       g-u_p\cdot \varphi_p(f)\in\m^{N+1+p},
     \end{displaymath}
     for all $p\geq 1$. The latter implies that the
     $u_p\cdot\varphi_p(f)$ converge to $g$ as well, and thus
     \begin{displaymath}
       g=u\cdot\varphi(f).
     \end{displaymath}

     Taking Lemma~\ref{lem:convergence} into account and using its
     terminology with $M=N-k\geq 1$ it suffices to construct certain
     series $b_{p,0}\in\m^{M+p-1}$ and $b_{p,i}\in\m^{M+p}$ for
     $i=1,\ldots,n$ and $p\geq 1$. 
     For this we note that by assumption
     \begin{displaymath}
       g-f\in\m^{N+1}=\m^{M-1}\cdot \m^{k+2}\subseteq \m^M\cdot\langle
       f\rangle+\m^{M+1}\cdot\jj(f). 
     \end{displaymath}
     Thus there are series $b_{1,0}\in\m^M$ and $b_{1,i}\in\m^{M+1}$
     for $i=1,\ldots,n$ such that
     \begin{equation}\label{eq:finitedeterminacybound:1}
       g-f=b_{1,0}\cdot f+\sum_{i=1}^n b_{1,i}\cdot f_{x_i}.
     \end{equation}
     As in Lemma~\ref{lem:convergence} we define $v_1=1+b_{1,0}\in\Kx^*$ and
     \begin{displaymath}
       \phi_1:\Kx\longrightarrow\Kx:x_i\mapsto x_i+b_{1,i}.
     \end{displaymath}
     We now show that
     \begin{displaymath}
       g-v_1\cdot \phi_1(f)\in\m^{N+2},
     \end{displaymath}
     since then we can replace $f$ in the above argument by
     $v_1\cdot\phi_1(f)$ and go on inductively.

     Note first that
     \begin{displaymath}
       (x_1+z_1)^{\beta_1}\cdot\ldots\cdot(x_n+z_n)^{\beta_n}
       =\sum_{\gamma_1=0}^{\beta_1}\ldots\sum_{\gamma_n=0}^{\beta_n}
       c_{\beta,\gamma}\cdot \ux^{\beta-\gamma}\cdot\uz^\gamma
     \end{displaymath}
     for
     $c_{\beta,\gamma}=\binom{\beta_1}{\gamma_1}\cdot\ldots\cdot\binom{\beta_n}{\gamma_n}\in\Z$. 
     For $f=\sum_{|\beta|\geq o}a_\beta\cdot \ux^\beta$ we thus have
     \begin{equation}
       \label{eq:taylorersatz}
       \begin{aligned}
         f(x_1+z_1,\ldots,x_n+z_n)=&
         \sum_{|\beta|\geq o}a_\beta\cdot
         (x_1+z_1)^{\beta_1}\cdot\ldots\cdot(x_n+z_n)^{\beta_n}\\
         =&\sum_{|\beta|\geq o}a_\beta\cdot
         \sum_{\gamma_1=0}^{\beta_1}\ldots\sum_{\gamma_n=0}^{\beta_n}
         c_{\beta,\gamma}\cdot \ux^{\beta-\gamma}\cdot\uz^\gamma
         =\sum_{\alpha\in\N^n} h_\alpha\cdot \uz^\alpha         
       \end{aligned}
     \end{equation}
     where
     \begin{displaymath}
       h_\alpha=\sum_{|\beta|\geq o,\beta\geq\alpha} a_\beta\cdot
       c_{\beta,\alpha}\cdot \ux^{\beta-\alpha}
     \end{displaymath}
     if we define $\beta\geq\alpha$ by $\beta_i\geq\alpha_i$ for all
     $i=1,\ldots,n$. It follows that
     \begin{displaymath}
       \ord(h_\alpha)=\min\big\{|\beta|-|\alpha|\;\big|\;|\beta|\geq
       o, \beta\geq\alpha\big\}\geq o-|\alpha|.
     \end{displaymath}
     We should like to point out that
     \begin{math}
       h_\alpha=\frac{D^\alpha f(\ux)}{\alpha_1!\cdot\ldots\cdot\alpha_n!}
     \end{math}
     whenever $\alpha_i<\Char(\K)$ for all $i=1,\ldots,n$. In
     particular the constant term $h_0=f$ and for the unit
     vectors $e_i$ we get $h_{e_i}=f_{x_i}$. 
     Applying $\phi_1$ to $f$ amounts to substituting $z_i$ by
     $b_{1,i}$ in \eqref{eq:taylorersatz} and we thus find
     \begin{displaymath}
       \phi_1(f)=f+\sum_{i=1}^n f_{x_i}\cdot b_{1,i}+ h
     \end{displaymath}
     where
     \begin{displaymath}
       h=\sum_{|\alpha|\geq 2} h_\alpha\cdot
       b_{1,1}^{\alpha_1}\cdots b_{1,n}^{\alpha_n}\in \m^{N+2},
     \end{displaymath}
     since
     \begin{align*}
       \ord\big(h_\alpha\cdot b_{1,1}^{\alpha_1}\cdots
       b_{1,n}^{\alpha_n}\big)\geq &\;
       \ord(h_\alpha)+\sum_{i=1}^n
       \ord(b_{1,i})\cdot\alpha_i\\
       \geq&\; o-|\alpha|+(M+1)\cdot |\alpha|
       \geq\; o+2\cdot M
       =N+2.
     \end{align*}
     Multiplying $\phi_1(f)$ by $v_1=1+b_{1,0}$ and using
     \eqref{eq:finitedeterminacybound:1} we get
     \begin{equation}\label{eq:finitedeterminacybound:2}
       \begin{aligned}
         g-v_1\cdot\phi_1(f)=&
         g-(1+b_{1,0})\cdot \left(f+\sum_{i=1}^n f_{x_i}\cdot
           b_{1,i}+h\right)\\
         =& -\sum_{i=1}^n b_{1,0}\cdot b_{1,i}\cdot f_{x_i}-
         (1+b_{1,0})\cdot h \in\m^{N+2},
       \end{aligned}
     \end{equation}
     since
     \begin{displaymath}
       \ord(b_{1,0}\cdot b_{1,i}\cdot f_{x_i})\geq M+(M+1)+(o-1)=N+2.
     \end{displaymath}
     We thus can proceed inductively to construct sequences
     $(b_{p,i})_{p\geq 1}$ for $i=0,\ldots,n$ with
     $b_{p,0}\in\m^{M+p-1}$ and $b_{p,i}\in\m^{M+p}$ for
     $i=1,\ldots,n$.  The generalisation of 
     \eqref{eq:finitedeterminacybound:2} holds by induction and with
     the notation of Lemma~\ref{lem:convergence} it reads as
     \begin{displaymath}
       g-u_p\cdot \varphi_p(f)\in\m^{N+1+p}
     \end{displaymath}
     as required. This finishes the proof for the contact equivalence.
     
     The proof for right equivalence works along the same lines. With
     the notation from above  the condition
     \begin{displaymath}
       \m^{k+2}\subseteq \m^2\cdot \jj(f)\subseteq \m^{o+1}
     \end{displaymath}
     implies that still $k\geq o-1$ and that for any $g$ with
     \begin{displaymath}
       g-f\in\m^{N+1}=\m^{M-1}\cdot \m^{k+2}\subseteq \m^{M+1}\cdot\jj(f)
     \end{displaymath}
     where $N\lang{=2k-\ord(f)+2}=2k-o+2\geq k+1$ and $M=N-k\geq 1$,
     there are
     $b_{1,i}\in\m^{M+1}$ with
     \begin{displaymath}
       g-f=b_{1,1}\cdot f_{x_1}+\ldots+b_{1,n}\cdot f_{x_n}.
     \end{displaymath}
     We can then define $\phi_1$ as above and see that
     \begin{displaymath}
       g-\phi_1(f)=h\in\m^{N+2}.
     \end{displaymath}
     Going on by induction and applying Lemma~\ref{lem:convergence} we
     get an automorphism $\varphi\in\Aut(\Kx)$ such that
     $g=\varphi(f)$.      
   \end{proof}
   
   \begin{lemma}\label{lem:convergence}
     Let $M\geq 1$ be an integer and let $b_{p,0}\in\m^{M+p-1}$ and $b_{p,i}\in\m^{M+p}$
     for $i=1,\ldots,n$ and $p\geq 1$. 
     Consider the units $v_p=1+b_{p,0}\in\Kx^*$ and the automorphisms
     $\phi_p\in\Aut(\Kx)$ given by
     \begin{displaymath}
       \phi_p:x_i\mapsto x_i+b_{p,i} \;\;\;\mbox{ for }\; i=1,\ldots,n.
     \end{displaymath}
     We denote by
     \begin{displaymath}
       \varphi_p=\phi_p\circ\phi_{p-1}\circ\ldots\circ\phi_{1}\in\Aut(\Kx)
     \end{displaymath}
     the composition of the first $p$ automorphisms, 
     and we define inductively
     \begin{displaymath}
       u_p=v_p\cdot\phi_p(u_{p-1}),
     \end{displaymath}
     where $u_0=1$.

     Then the following hold true:
     \begin{enumerate}
     \item The sequences $\big(\varphi_p(x_i)\big)_{p\geq 1}$ converge in
       the $\m$-adic topology of $\Kx$ to  power series $x_i+b_i$ with
       $b_i\in\m^{M+1}$ for $i=1,\ldots,n$. In particular, the map
       \begin{displaymath}
         \varphi:\Kx\longrightarrow\Kx:x_i\mapsto x_i+b_i
       \end{displaymath}
       is a local $\K$-algebra automorphism of $\Kx$.
     \item The sequence
       $(u_p)_{p\geq 1}$ converges in the $\m$-adic topology to a
       unit $u=1+b_0\in\Kx^*$ with $b_0\in\m^M$.
     \item For any power series $f_0\in\Kx$ the sequence
       $\big(\varphi_p(f_0)\big)_{p\geq 1}$ converges in the $\m$-adic topology
       to $\varphi(f_0)$.
     \item For any power series $f_0\in\Kx$ the sequence
       $\big(u_p\cdot\varphi_p(f_0)\big)_{p\geq 1}$ converges in the $\m$-adic topology
       to $u\cdot\varphi(f_0)$.
     \end{enumerate}
   \end{lemma}
   \begin{proof}
     Since $b_{p,i}\in\m^{M+p}$ for $i=1,\ldots,n$ we have by construction that
     \begin{displaymath}
       \varphi_p(x_i)-\varphi_{p-1}(x_i)=
       \phi_p\big(\varphi_{p-1}(x_i)\big)-\varphi_{p-1}(x_i)\in \m^{M+p},
     \end{displaymath}
     and thus for any $N\geq 1$ there is a $P=\max\{N-M,1\}\geq 1$ such that for all
     $p>q> P$
     \begin{displaymath}
       \varphi_p(x_i)-\varphi_q(x_i)=\sum_{j=q+1}^p
       \varphi_j(x_i)-\varphi_{j-1}(x_i)
       \in \m^{M+P}\subseteq\m^N.
     \end{displaymath}
     This shows that the $\varphi_p(x_i)$ converge to a power series of
     the form $x_i+b_i$ with $b_i\in\m^{M+1}$. \lang{Note
       that $b_i\equiv b_{i,1}\;(\mod \m^{M+1})$.} 
     Similarly we have that
     \begin{displaymath}
       \phi_p(u_{p-1})-u_{p-1}\in \m^{M+p},
     \end{displaymath}
     and since $b_{p,0}\in\m^{M+p-1}$ thus also
     \begin{displaymath}
       u_p-u_{p-1}=(1+b_{p,0})\cdot \phi_p(u_{p-1})-u_{p-1}\in \m^{M+p-1}.
     \end{displaymath}
     With basically the same argument as above we see that $u_p$ converges in the
     $\m$-adic topology to a
     power series of the form $1+b_0$ with $b_0\in\m^M$. \lang{Note
       that $b_0\equiv b_{0,1}\;(\mod \m^M)$.}

     Let now $f_0\in\Kx$ be any power series and let $N\in\N$ be given. 
     Since the $\varphi_p(x_i)$ converge to
     $\varphi(x_i)$ and the $u_p$ converge to $u$, there is a $P\geq
     1$ such that for all $p\geq P$ and $i=1,\ldots,n$
     \begin{displaymath}
       \varphi(x_i)-\varphi_p(x_i)\in\m^N
       \;\;\;\mbox{ as well as }\;\;\;
       u-u_p\in\m^N.
     \end{displaymath}
     It follows that also
     \begin{displaymath}
       \varphi(f_0)-\varphi_p(f_0)\in\m^N, \kurz{ \text{ and } }
     \end{displaymath}
     \lang{and}
     \begin{displaymath}
       u\cdot \varphi(f_0)-u_p\cdot\varphi_p(f_0)
       =u\cdot \big(\varphi(f_0)-\varphi_p(f_0)\big)
       +(u-u_p)\cdot\varphi_p(f_0)\in\m^N
     \end{displaymath}
     for all $p\geq P$. Thus the $\varphi_p(f_0)$ converge to
     $\varphi(f_0)$ and the $u_p\cdot \varphi_p(f_0)$ converge to
     $u\cdot\varphi(f_0)$. 
   \end{proof}

   \begin{remark}\label{rem:finitedeterminacybound}
     \begin{enumerate}
     \item      If the base field $\K$ has characteristic zero it is known that
       \bmath
         \m^{k+2}\subseteq \m^2\cdot \jj(f)
       \emath
       implies right-$(k+1)$-determinacy of $f$, and that
       \bmath
         \m^{k+2}\subseteq \m\cdot\langle f\rangle+\m^2\cdot\jj(f)
       \emath
       implies contact-$(k+1)$-determinacy of $f$ -- see e.g.\ 
       \cite[Thm.~2.23]{GLS07} for $\K=\C$ and \cite[Thm.~3.1.13]{Bou09} for the
       general case. \lang{In particular, the 
       assumptions of Theorem~\ref{thm:finitedeterminacybound} would in
       each case imply $k+1$ as determinacy bound, which in general is
       strictly better than 
       $2k-\ord(f)+2\geq k+1$. }
     \item In positive characteristic the bounds in (a) do
       not hold any longer. Consider the power series
       $f=y^2+x^3y\in\K[x,y]$, $\Char(\K)=2$. Then
       $\tj(f)=\langle y^2,x^2y,x^3\rangle$ and thus $\tau(f)=5$. In
       particular, $f$ defines an isolated hypersurface singularity
       $R_f$. Moreover, we have
       \begin{displaymath}
         \langle f\rangle+\m\cdot\jj(f)=\langle y^2,x^3y,x^4\rangle\supset\m^4,
       \end{displaymath}
       and if (a) would hold then $f$ would be
       $5$-determined. However, $f$ is reducible while
       $f+x^5$ is irreducible as can be checked by the procedure
       \texttt{is\_irred} in \textsc{Singular} \cite{DGPS10}. Therefore, $R_f$ and
       $R_{f+x^5}$ cannot be 
       isomorphic, i.e.~ $f\not\sim_c f+x^5$, and $f$ is not
       contact $5$-determined. 
       Theorem~\ref{thm:finitedeterminacybound} asserts that $f$ is
       actually 
       $6$-determined, i.e.\ our result is sharp in this example.
     \item The determinacy bounds given in Theorem~\ref{thm:finitedeterminacybound} are
       always at least as good as the previously known bounds $2\cdot\mu(f)$
       respectively $2\cdot\tau(f)$ for arbitrary characteristic, and
       they are in general much better (see e.g.\ the example in Part (b)).
       \kurz{This follows from $\m^{\mu(f)}\subseteq\jj(f)$ if
         $\mu(f)<\infty$ and $\m^{\tau(f)}\subseteq\tj(f)$ if
         $\tau(f)<\infty$.}
       \lang{This follows from the following two facts, which are easy
       consequences of Nakayama's Lemma:
       \begin{itemize}
       \item If $\mu(f)<\infty$, then $\m^{\mu(f)}\subseteq\jj(f)$.
       \item If $\tau(f)<\infty$, then $\m^{\tau(f)}\subseteq\tj(f)$.
       \end{itemize}
       To see this it suffices to show that for any ideal $I\unideal\Kx$ with
       $d=\dim_{\K}(\Kx/I)<\infty$ we have $\m^d\subseteq I$. In the
       ring $\Kx/I$ by Nakayama's Lemma the descending sequence of
       powers $\overline{\m}^k$ of the maximal ideal has to be
       strictly descending until it finally becomes zero. Thus in each
       step the codimension of $\overline{\m}^k$ grows at least by
       one, and it can do so at most $d$ times. Thus
       $\overline{m}^d=0$ or equivalently $\m^d\subseteq I$.}
     \item In concrete examples the integers $k$ in
       Theorem~\ref{thm:finitedeterminacybound} can be computed in 
       \textsc{Singular} with the aid of the procedure
       \texttt{highcorner}. If we apply \texttt{highcorner} to a
       standard basis of the ideal $\m^2\cdot\jj(f)$ resp.\ $\m\cdot\langle
       f\rangle+\m^2\cdot\jj(f)$ with respect to some local degree
       ordering the result will be a monomial 
       $\ux^\alpha$, and then $k=\deg(\ux^\alpha)-1$. E.g.\ for
       $f=y^8+x^8y^4+x^{23}$ and $\Char(\K)=23$ the following
       \textsc{Singular} computation shows that
       $k=\deg(x^{22}y^2)-1=23$ and $f$ is at least contact $40$-determined.
\begin{verbatim}
> ring r=23,(x,y),ds;
> poly f=y8+x8y4+x23;
> ideal I=maxideal(1)*f+maxideal(2)*jacob(f);
> I=std(I);
> highcorner(I);
x22y2
\end{verbatim}

     \end{enumerate}
   \end{remark}

   \begin{corollary}\label{cor:finitedeterminacybound}
     Let $0\not=f\in\m^2\unideal \Kx$.
     \begin{enumerate}
     \item If $\mu(f)<\infty$, then the right determinacy of $f$ is at
       most $2\mu(f)-\ord(f)+2$.
     \item If $\tau(f)<\infty$, then the contact determinacy of $f$ is
       at most $2\tau(f)-\ord(f)+2$. 
     \end{enumerate}            
   \end{corollary}
   \begin{proof}
     This follows from Remark~\ref{rem:finitedeterminacybound}
     (c) and Theorem~\ref{thm:finitedeterminacybound}.
   \end{proof}

   The assumptions in Theorem~\ref{thm:finitedeterminacybound} are
   fulfilled if $f$ respectively $R_f$ is an isolated singularity, and
   thus these are finitely 
   determined. In the complex setting it is well known that the
   converse holds as well (see e.g.\ \cite[Cor.~2.39]{GLS07}),
   and the same is true in arbitrary characteristic.

   \begin{theorem}\label{thm:isolated}
     Let $0\not=f\in\m\unideal\Kx$ be a power series.
     \begin{enumerate}
     \item If $f$ is right $k$-determined, then
       $\m^{k+1}\subseteq\m\cdot\jj(f)$. In particular, $f$ is an
       isolated singularity.
     \item If $f$ is contact $k$-determined, then
       $\m^{k+1}\subseteq\langle f\rangle+\m\cdot\jj(f)$. In
       particular, $R_f$ is an isolated hypersurface singularity.
     \end{enumerate}
   \end{theorem}

   In the proof of Theorem~\ref{thm:isolated} we will restrict
   ourselves to the case of contact equivalence, since the case of
   right equivalence can be treated analogously. But before we come to
   the proof we would like to fix some notation. We have already seen
   that contact equivalence can be phrased via the action of the
   contact group $\mathcal{K}$ on $\Kx$. If a power series is finitely
   determined then 
   only terms up to some finite order are relevant. We
   thus consider $\Kx$ as well as the contact
   group modulo some power of the maximal ideal\lang{, and we will now
   introduce the necessary notation}.

   We denote by
   \begin{displaymath}
     J_l=\Kx/\m^{l+1}
   \end{displaymath}
   the space of $l$-jets of power series in $\Kx$. 
   Recall that each local $\K$-algebra automorphism $\varphi$ of $\Kx$
   is uniquely represented by a tuple $(\varphi_1,\ldots,\varphi_n)\in
   \Kx^n$ of power series such that
   \bmath
     \varphi_i(0)=0 \;\;\;\mbox{ for all } i=1,\ldots,n
   \emath
   and
   \begin{displaymath}
     \det\left(\frac{\partial \varphi_i}{\partial x_j}(0)\right)_{i,j=1,\ldots,n}\not=0.
   \end{displaymath}   
   We define the $l$-jet of the automorphism
   $\varphi$ as
   \bmath
     \jet_l(\varphi):=\left(\jet_l(\varphi_1),\ldots,\jet_l(\varphi_n)\right),
   \emath
   and the $l$-jet of the contact group
   \begin{displaymath}
     \mathcal{K}_l:=\jet_l(\Kx^*)\ltimes\jet_l\big(\Aut(\Kx)\big)
   \end{displaymath}
   via the multiplication
   \begin{displaymath}
     \big(\jet_l(u),\jet_l(\varphi)\big)\cdot\big(\jet_l(v),\jet_l(\psi)\big)
     :=\big(\jet_l(u\cdot \varphi(v)),\jet_l(\varphi\circ\psi)\big)
   \end{displaymath}
   which is independent of the chosen
   representatives. The $l$-jet of the contact group then operates on
   the $l$-jet of $\Kx$ via
   \begin{displaymath}
     \Phi_l:\mathcal{K}_l\times J_l\longrightarrow J_l:
     \big((\jet_l(u),\jet_l(\varphi)),\jet_l(f)\big)
     \mapsto \jet_l\big(u\cdot \varphi(f)\big),
   \end{displaymath}
   i.e.\ by taking representatives, let them act and taking the
   $l$-jet. 

   Analogously, we define the $l$-jet $\mathcal{R}_l$ of the right
   group $\mathcal{R}=\Aut(\Kx)$ and it operates on $J_l$. 

   \begin{remark}
     $J_l$ is an affine space and $\mathcal{K}_l$ and
     $\mathcal{R}_l$ are affine algebraic groups acting on $J_l$
     via a regular \emph{separable} algebraic action.     
   \end{remark}
   \lang{\begin{proof}[Proof that the actions are separable]
       We}\kurz{To see that the action is separable we}
     restrict to the case of the action of $\mathcal{K}_l$ and
     we choose coordinates on $J_l$ and of $\mathcal{K}_l$.
     Writing
     \bmath
       \jet_l(f)=\sum_{|\alpha|=0}^l a_\alpha \overline{\ux}^\alpha,
     \emath
     \lang{and}
     \bmath
       \jet_l(\varphi_i)=\sum_{|\beta|=1}^l b_{i,\beta} \overline{\ux}^\beta,
     \emath
     and
     \bmath
       \jet_l(u)=\sum_{|\gamma|=0}^l c_\gamma \overline{\ux}^\gamma
     \emath
     we have the coordinates
     \bmath
       (a_\alpha,b_{i,\beta},c_\gamma)_{\alpha,i,\beta,\gamma}
     \emath
     on $\mathcal{K}_l\times J_l$ with $c_0\not=0$
     and $\det(B)\not=0$
     where $B=(B_{ij})$ with $B_{ij}=\frac{\partial
       \varphi_i}{\partial x_j}(0)=b_{i,e_j}$ and $e_j$ \lang{is} the 
     $j$-th canonical basis vector in $\Z^n$. Using in the same manner
     the coordinates
     \bmath
       (a_\delta')_{|\delta|=0,\ldots,l}
     \emath
     on the target space the action is given by polynomial maps
     \begin{displaymath}
       a_\delta'=F_\delta(a_\alpha,b_{i,\beta},c_\gamma),
     \end{displaymath}
     and it is important to note that the inverse of the action is
     given by rational maps
     \begin{displaymath}
       a_\alpha=\frac{G_\alpha(a_\delta',b_{i,\beta},c_\gamma)}{H_\alpha(a_\delta',b_{i,\beta},c_\gamma)}.
     \end{displaymath}
     The reason for this is that we can solve for the $a_\alpha$
     degree by degree starting basically with Cramer's rule, and
     this property ensures that for the extension of the fields of rational
     functions induced by the operation $\Phi_l$ we have
     \begin{displaymath}
       \K(J_l)=\K(a_\delta')\subset \K(\mathcal{K}_l\times
       J_l)=\K(a_\alpha,b_{i,\beta},c_\gamma)
       =\K(a_\delta',b_{i,\beta},c_\gamma)=\K(J_l)(b_{i,\beta},c_\gamma).
     \end{displaymath}
     The $b_{i,\beta}$ and $c_\gamma$ are algebraically independent
     over $\K(a_\alpha)$ and comparing transcendence degrees they must
     be so over $\K(J_l)$. Thus $\K(\mathcal{K}_l\times J_l)$ is a
     purely transcendental extension of $\K(J_l)$, and it is thus by
     default a separably generated extension in the sense of
     \cite[p.~27]{Har77}. Thus $\mathcal{K}_l$ operates separably on
     $J_l$. 
   \kurz{\hfill$\Box$}\lang{\end{proof}} 

   This allows us to describe the tangent space to the orbits also in
   positive characteristic (for $\K=\C$ see \cite[Prop.~2.38]{GLS07}).

   \begin{proposition}
     Let $f\in\Kx$. Then the tangent space to the orbit of $\jet_l(f)$
     under the action of $\mathcal{R}_l$ respectively $\mathcal{K}_l$
     considered as a subspace of $J_l$ is
     \begin{displaymath}
       T_{\jet_l(f)}\big(\mathcal{R}_l\cdot \jet_l(f)\big)=
       \big(\m\cdot\jj(f)+\m^{l+1}\big)/\m^{l+1}
     \end{displaymath}
     respectively
     \begin{displaymath}
       T_{\jet_l(f)}\big(\mathcal{K}_l\cdot \jet_l(f)\big)=
       \big(\langle f\rangle+\m\cdot\jj(f)+\m^{l+1}\big)/\m^{l+1}.
     \end{displaymath}
   \end{proposition}
   \begin{proof}
     If $G$ denotes one of the two above groups then the action of $G$
     on $J_l$ induces a surjective separable morphism $G\longrightarrow G\cdot
     \jet_l(f)$ of smooth varieties. Thus the induced differential map on the
     tangent spaces is generically surjective (see e.g.\ the
       proof of \cite[Lem.~III.10.5.1]{Har77}). \lang{However,
         since}\kurz{Since} we can 
     translate each point in $G$ to a point where the tangent map is
     surjective and the translation is an isomorphism the differential
     map is actually always surjective. It thus suffices to understand the
     image of the tangent space to $G$ at the neutral element of the
     group and its image under the differential map. We restrict here
     to the case $G=\mathcal{K}_l$ since the proof for $\mathcal{R}_l$
     works analogously.

     The tangent space to $\mathcal{K}_l$ at $(1,\id)$ can be
     described via the local $\K$-algebra homomorphisms from the local
     ring of $\mathcal{K}_l$ to $\K[\varepsilon]$ where
     $\varepsilon^2=0$. In this sense, a tangent vector is represented
     by the residue class modulo $\m^{l+1}$ of a tuple
     \begin{displaymath}
       \big(1+\varepsilon\cdot a,\id+\varepsilon\cdot \phi\big)
     \end{displaymath}
     with $a\in\Kx$ and $\phi=(\phi_1,\ldots,\phi_n)$ where
     $\phi_i\in\m$. We now apply the differential map by acting with
     the above tuple on $f$ modulo $\m^{l+1}$. Taking
     $\varepsilon^2=0$  into account and expanding the power series as
     in \eqref{eq:taylorersatz} we get
     \begin{displaymath}
       (1+\varepsilon\cdot a)\cdot
       f\big(\ux+\varepsilon\cdot\phi\big)
       =f+\varepsilon\cdot \left(a\cdot f+ \sum_{i=1}^n f_{x_i} \cdot \phi_i\right).
     \end{displaymath}
     Interpreted in $J_l$ this tangent vector is just the $l$-jet of     
     \begin{displaymath}
       a\cdot f+ \sum_{i=1}^n f_{x_i} \cdot \phi_i,
     \end{displaymath}
     which proves the claim\tom{ (see \cite[Lem.~~7.4]{Gre98})}.
   \end{proof}

   \begin{proof}[Proof of Theorem \ref{thm:isolated}]
     We only do the proof for contact determinacy since the other case
     works analogously.
     If $f$ is contact $k$ determined and $g\in \m^{k+1}$ then for any
     $t\in \K$ the $k+1$-jet $\jet_{k+1}(f)+t\cdot \jet_{k+1}(g)$ is in the orbit
     of $\jet_{k+1}(f)$ under  $\mathcal{K}_{k+1}$. But then
     \tom{
       \begin{displaymath}
         \jet_{k+1}(f)+\varepsilon\cdot\jet_{k+1}(g)
       \end{displaymath}
       is in the tangent space of the orbit of $J_{k+1}$ if we interpret
       the tangent space via local $\K$-algebra homomorphisms, and thus }
     \begin{displaymath}
       \jet_{k+1}(g)\in T_{\jet_{k+1}(f)}\big(\mathcal{K}_{k+1}\cdot \jet_{k+1}(f)\big)=
       \big(\langle f\rangle+\m\cdot\jj(f)+\m^{k+2}\big)/\m^{k+2}.
     \end{displaymath}
     This implies
     \begin{displaymath}
       g\in \langle f\rangle+\m\cdot\jj(f)+\m^{k+2},
     \end{displaymath}
     and hence
     \begin{displaymath}
       \m^{k+1}\subseteq \langle f\rangle+\m\cdot\jj(f)+\m^{k+2}.
     \end{displaymath}
     By Nakayama's Lemma we get
     \bmath
       \m^{k+1}\subseteq \langle f\rangle+\m\cdot\jj(f),
     \emath
     as claimed.
   \end{proof}

   Combining the results of Corollary~\ref{cor:finitedeterminacybound}
   and of Theorem~\ref{thm:isolated} we obtain\kurz{:}\lang{ the following result.}

   \begin{theorem}\label{thm:fd=isolated}
     Let $0\not=f\in\m\ideal\Kx$ be a power series.
     \begin{enumerate}
     \item $f$ is an isolated singularity if and only
       if $f$ is finitely right determined.
     \item $R_f$ is an isolated hypersurface singularity if and only
       if $f$ is finitely contact determined.
     \end{enumerate}     
   \end{theorem}


   \section{Non-degenerate singularities}\label{sec:nd}

   \lang{For the convenience of the reader let }\kurz{Let } us recall the definition of
   the Newton diagram and Wall's notion of a $C$-polytope (see \cite{Wal99}). 
   To each power series $f=\sum_\alpha a_\alpha\ux^\alpha\in\Kx$ we
   associate its \emph{Newton polyhedron}
   $\Gamma_+(f)$ as the convex hull of the set
   \begin{displaymath}
     \bigcup_{\alpha\in\supp(f)}\big(\alpha+\R_{\geq 0}^n\big)
   \end{displaymath}
   where $\supp(f)=\{\alpha\;|\;a_\alpha\not=0\}$ denotes the
   support of $f$. This is an unbounded polyhedron in
   $\R^n$. Following the convention of Arnol'd we call
   the union $\Gamma(f)$ of its compact faces the \emph{Newton
     diagram} of $f$, some authors call it the \emph{Newton polytope}
   resp.~ the \emph{Newton polygon} if $n=2$. By $\Gamma_-(f)$ we denote the union of all
   line segments joining the origin to a point on $\Gamma(f)$. (See
   Figure~\ref{fig:np} for an example.)      
   \begin{figure}[h]
     \centering
     \setlength{\unitlength}{0.35mm}
     \begin{tabular}{ccc}
       \begin{picture}(100,80)
         \put(10,10){\line(1,0){100}}
         \put(10,20){\line(1,0){100}}
         \put(10,30){\line(1,0){100}}
         \put(10,40){\line(1,0){100}}
         \put(10,50){\line(1,0){100}}
         \put(10,60){\line(1,0){100}}
         \put(10,70){\line(1,0){100}}         
         \put(10,10){\line(0,1){70}}
         \put(20,10){\line(0,1){70}}
         \put(30,10){\line(0,1){70}}
         \put(40,10){\line(0,1){70}}
         \put(50,10){\line(0,1){70}}
         \put(60,10){\line(0,1){70}}
         \put(70,10){\line(0,1){70}}
         \put(80,10){\line(0,1){70}}
         \put(90,10){\line(0,1){70}}
         \put(100,10){\line(0,1){70}}
         \thicklines\drawline[12](20,50)(40,30)
         \thicklines\drawline[12](40,30)(80,10)         
         \linethickness{0.2mm}\scriptsize
         \Thicklines
         \put(20,50){\line(0,1){30}}
         \put(80,10){\line(1,0){30}}
         \put(20,70){\line(1,1){10}}
         \put(20,60){\line(1,1){20}}
         \put(20,50){\line(1,1){30}}
         \put(25,45){\line(1,1){35}}
         \put(30,40){\line(1,1){40}}
         \put(35,35){\line(1,1){45}}
         \put(40,30){\line(1,1){50}}
         \put(46.5,26.5){\line(1,1){52.5}}
         \put(53.5,23.5){\line(1,1){55.5}}
         \put(60,20){\line(1,1){50}}
         \put(66.5,16.5){\line(1,1){43.5}}
         \put(72.5,12.5){\line(1,1){37}}
         \put(80,10){\line(1,1){30}}
         \put(90,10){\line(1,1){20}}
         \put(100,10){\line(1,1){10}}
         
         \put(20,50){\circle*{4}}
         \put(30,40){\circle*{4}}
         \put(40,30){\circle*{4}}
         \put(80,10){\circle*{4}}
         \put(50,30){\circle{4}}
         \put(35,-7){\mbox{${\Gamma}_{_+}(f)$}}
       \end{picture}
       &
       \begin{picture}(100,80)
         \put(10,10){\line(1,0){90}}
         \put(10,20){\line(1,0){90}}
         \put(10,30){\line(1,0){90}}
         \put(10,40){\line(1,0){90}}
         \put(10,50){\line(1,0){90}}
         \put(10,60){\line(1,0){90}}
         \put(10,70){\line(1,0){90}}
         \put(10,10){\line(0,1){70}}
         \put(20,10){\line(0,1){70}}
         \put(30,10){\line(0,1){70}}
         \put(40,10){\line(0,1){70}}
         \put(50,10){\line(0,1){70}}
         \put(60,10){\line(0,1){70}}
         \put(70,10){\line(0,1){70}}
         \put(80,10){\line(0,1){70}}
         \put(90,10){\line(0,1){70}}
         \thicklines\drawline[12](20,50)(40,30)
         \thicklines\drawline[12](40,30)(80,10)
         \put(20,50){\circle*{4}}
         \put(30,40){\circle*{4}}
         \put(40,30){\circle*{4}}
         \put(80,10){\circle*{4}}
         \put(25,-7){\mbox{ ${\Gamma(f)}$}}
       \end{picture}
       &
       \begin{picture}(100,80)
         \put(10,10){\line(1,0){90}}
         \put(10,20){\line(1,0){90}}
         \put(10,30){\line(1,0){90}}
         \put(10,40){\line(1,0){90}}
         \put(10,50){\line(1,0){90}}
         \put(10,60){\line(1,0){90}}
         \put(10,70){\line(1,0){90}}
         \put(10,10){\line(0,1){70}}
         \put(20,10){\line(0,1){70}}
         \put(30,10){\line(0,1){70}}
         \put(40,10){\line(0,1){70}}
         \put(50,10){\line(0,1){70}}
         \put(60,10){\line(0,1){70}}
         \put(70,10){\line(0,1){70}}
         \put(80,10){\line(0,1){70}}
         \put(90,10){\line(0,1){70}}
         \thicklines\drawline[12](20,50)(40,30)
         \thicklines\drawline[12](40,30)(80,10)
         \thicklines\drawline[12](10,10)(20,50)
         \linethickness{0.2mm}\scriptsize
         \Thicklines
         \put(10,10){\line(1,0){70}}
         \put(70,10){\line(1,1){3.3}}
         \put(60,10){\line(1,1){6.7}}
         \put(50,10){\line(1,1){10}}
         \put(40,10){\line(1,1){13.3}}
         \put(30,10){\line(1,1){16.7}}
         \put(20,10){\line(1,1){20}}
         \put(10,10){\line(1,1){25}}
         \put(13.1,23.1){\line(1,1){17}}
         \put(17,37){\line(1,1){8}}
         \put(20,50){\circle*{4}}
         \put(30,40){\circle*{4}}
         \put(40,30){\circle*{4}}
         \put(80,10){\circle*{4}}
         \put(35,-7)
         {\mbox{ ${\Gamma}_{_{-}}(f)$}}
       \end{picture}
     \end{tabular}
     
     \caption{The Newton diagram of $x\cdot(y^4+xy^3+x^2y^2-x^3y^2+x^6)$.}
     \label{fig:np}
   \end{figure}
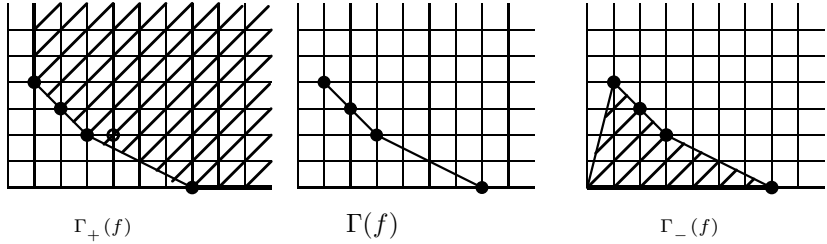
   
   If the Newton diagram of a singularity $f$ meets all coordinate axes
   we call $f$ \emph{convenient}. In this case the Newton diagram
   of $f$ can be used to define a filtration on $\Kx$ by finite
   dimensional vector spaces. However, not
   every isolated singularity is convenient, and one then has to
   enlarge the Newton diagram. A compact rational polytope $P$ of
   dimension $n-1$ in the positive orthant $\R_{\geq 0}^n$ is called a
   $C$-polytope if the region above $P$ is 
   convex and if every ray in the positive orthant emanating
   from the origin meets $P$ in exactly one point. The Newton diagram
   $\Gamma(f)$ is a $C$-polytope if and only if $f$ is convenient.

   \smallskip

   We will now first introduce the different notions of
   non-degeneracy.
   For this let $f=\sum_\alpha a_\alpha\cdot \ux^\alpha\in\m$ be a power
   series, let $P$ be a $C$-polytope such that $\supp(f)$ has \emph{no}
   point below $P$, and let $\Delta$ be a face of $P$. 
   By $\In_\Delta(f)=\sum_{\alpha\in\Delta}a_\alpha\cdot\ux^\alpha$
   we denote the \emph{initial form} or \emph{principal part} of $f$ along
   $\Delta$. 

   Following Wall we
   call $f$ \emph{non-degenerate \ND\ along
     $\Delta$} if the Jacobian ideal $\jj(\In_\Delta(f))$ has no zero in the torus
   $(\K^*)^n$. $f$ is then said to be \emph{Newton non-degenerate}
   \NND\ if $f$ is non-degenerate along each
   \emph{face} of the Newton diagram $\Gamma(f)$. Note that \lang{unlike
     Wall} we do \emph{not} require $f$ to be \emph{convenient}.

   To define inner non-degeneracy we need to fix two more notions.
   The face $\Delta$ is an \emph{inner face} of $P$ if it is not
   contained in any coordinate hyperplane.
   And each point $q\in\K^n$ determines a coordinate hyperspace
   $H_q=\bigcap_{q_i=0}\{x_i=0\}\subseteq\R^n$ in $\R^n$.
   We call $f$ \emph{inner non-degenerate \SND\ along
     $\Delta$} if for each zero $q$ of the Jacobian ideal $\jj(\In_\Delta(f))$ the
   polytope $\Delta$ contains no point on $H_q$.
   Then $f$ is called \emph{inner Newton non-degenerate} \SNND\
   w.r.t.\ a $C$-polytope $P$ if $f$ is inner non-degenerate along each
   \emph{inner} face of $P$ and $\supp(f)$ has no point below $P$.
   We say that $f$ satisfies \SNND\ if there exists a $C$-polytope $P$
   such that $f$ is \SNND\ w.r.t.\ $P$.

   Finally, we call $f$ \emph{weakly non-degenerate \WND\ along
     $\Delta$} if the Tjurina ideal $\tj(\In_\Delta(f))$ has no zero in the torus
   $(\K^*)^n$,
   and $f$ is called \emph{weakly Newton non-degenerate} \WNND\ if $f$ is
   weakly non-degenerate along each
   \emph{facet} of $\Gamma(f)$. Recall that a facet is a
   top-dimensional face.

   Non-degenerate singularities as introduced above are interesting
   since for these the Newton diagram can be used to compute
   invariants combinatorially as we show below and moreover most singularities are
   non-degenerate (see Remark \ref{rem:nd} (j)).

   \begin{remark}\label{rem:nd}
     We collect some easy facts on and relations between the
     different types of non-degeneracy. For any occuring $C$-polytope
     and power series $f$ we assume that no point in $\supp(f)$ lies
     below $P$.
     \begin{enumerate}
     \item Each of the non-degeneracy conditions
       introduced above only depends the principal
       part $\In_P(f)=\sum_{\alpha\in P}a_\alpha\cdot\ux^\alpha$ of
       $f$ w.r.t.\ $P$.
     \item Obviously \ND\ along $\Delta$ implies \WND\ along $\Delta$
       and both are equivalent in characteristic zero, or, more
       generally, if $\Char(\K)$ does not divide the weighted degree of
       $\In_\Delta(f)$. 
     \item \WNND\ along $\Delta$ is strictly weaker than \NND\ along
       $\Delta$ in positive characteristic, but they are equivalent in
       characteristic zero. Moreover, \WNND\ imposes only conditions on
       the facets of $\Gamma(f)$ while \NND\ does so for all faces of
       any dimension.\\ E.g.\ $f=x^3+y^2$ with
       $\Char(\K)=3$ is \WNND\ but not \NND, since $f$ is
       not \ND\ along $\Delta=\{(3,0)\}$.
     \item If $f$ is \SND\ along $\Delta$, then $f$ is \ND\ along
       $\Delta$, but the converse is not true in general.\\
       \laenger{ For this note that $q$ is a zero of $\jj(\In_\Delta(f))$ in
         the torus, then $H_q=\R^n$ intersects $\Delta$. }
       E.g.\ $f=x^2y^2+y^4$
       and $\Delta$ the line segment from $(4,0)$ to $(0,4)$, then $f$
       satisfies \ND\ along $\Delta$, but not \SND\laenger{, since
         $q=(1,0)\in\Sing(\In_\Delta(f))$ and $H_q\cap\Delta=\{(4,0)\}$}.
     \item If $\Delta$ does not meet any coordinate hyperplane, then  $f$
       is \ND\ along $\Delta$ if and only if $f$ is \SND\ along $\Delta$. 
       \laenger{ 
         For this note that $\Delta\cap H_q\not=\emptyset$ by
         assumption implies $q\in(K^*)^n$.}
     \item By (d) \NND\ does not imply \SNND\, but also \SNND\ need
       not imply \NND\ since it only imposes conditions on the inner
       faces, see (g).
     \item $f$ can be convenient and \SNND\ without satisfying
       \NND.\\ E.g. $f=(x+y)^2+xz+z^2$ with $\Char(\K)\not=2$, then
       $\Gamma(f)$ has a unique facet $\Delta$ with $f=\In_\Delta(f)$
       and $\Sing(f)=\{0\}$. Thus $f$ is \SNND\ since no other face is
       inner. But $f$ is not \ND\
       along the line segment from $(2,0,0)$ to $(0,2,0)$ which is a
       face of $\Gamma(f)$ (see also \cite{Kou76}).
     \item If $f$ is \NND\ and $k\cdot e_i\in\Gamma(f)$, where $e_i$
       is the $i$-th standard basis vector of $\R^n$, then
       $\Char(\K)$ does not divide $k$.
       \laenger{
         
         For this note that $\Delta=\{k\cdot e_i\}$ is a face of
         $\Gamma(f)$ and $\In_\Delta(f)=a\cdot x_i^k$ for some
         non-zero constant $a\in\K^*$. By assumption 
         $\jj(\In_\Delta(f))$ has no zero in the torus, so the characteristic of
         $\K$ cannot divide $k$.}       
     \item $f$ satisfies  \SND\ at an inner vertex
       $\alpha=(\alpha_1,\ldots,\alpha_n)$ of $P$ if and only if
       $\alpha$ is a vertex of $\Gamma(f)$ and some $\alpha_i$ is not
       divisible by $\Char(\K)$.
       \laenger{ 

         For this set $\Delta=\{\alpha\}$, so that
         $\In_\Delta(f)=a_\alpha\cdot \ux^\alpha$ has a non-zero
         partial derivative if and only if $\alpha\in\Gamma(f)$ and
         there is some $i$ such that $\Char(\K)$ does not divide
         $\alpha_i$, and this is equivalent to \SND\ since $\Delta$
         does not intersect any coordinate space. Note that if
         $\alpha$ is in $\Gamma(f)$ then it is necessarily a vertex
         since no point of $\supp(f)$ lies below $P$ and $\alpha$ was
         assumed to be a vertex of $P$.}
     \item In characteristic zero each of the above non-degeneracy
       conditions is a generality condition \laenger{(see
         \cite[Thm.~I]{Kou76} and \cite[Thm.~3.1]{Wal99})} in the sense that fixing a
       $C$-polytope $P$ then, among all polynomials
       $f$ with $\supp(f)\subseteq P$, there is a Zariski open dense
       subset which satisfies the non-degeneracy condition. In
       positive characteristic some additional assumptions on the $C$-polytope
       $P$ are necessary, like that not all coordinates of a vertex
       should be divisible by the characteristic. 
     \end{enumerate}
     \lang{\hfill$\Box$}
   \end{remark}

   The following remark sheds some light on the definition of
   \NND.

   \begin{remark}[\cite{Kou76}]
     Kouchnirenko defines \ND\ and \NND\ 
     by considering common zeros of $x_i\cdot
     \In_\Delta(f)_{x_i}$ for $i=1,\ldots,n$, since these
     polynomials are better suited with respect to the piecewise
     filtration induced by $\Gamma(f)$ (see
     \cite{Kou76} or \cite[Sec.~3]{BGM10a} for details on this filtration).
     However, they have no common zero
     in the torus if and only if $\jj(\In_\Delta(f))$ has no zero
     in the torus, so that the two definitions
     coincide. 

     Each face $\Delta$ of the Newton diagram of a convenient power
     series $f$ determines
     a finitely generated semigroup $C_\Delta$ in $\Z^n$ by considering those
     lattice points which lie in the cone over $\Delta$ with the
     origin as base. This semigroup then determines a finitely
     generated $\K$-algebra
     $\K[C_\Delta]=\K[\ux^\alpha|\alpha\in C_\Delta]$, and
     the polynomials $x_i\cdot \In_\Delta(f)_{x_i}$, $i=1,\ldots,n$
     generate an ideal, say $I_\Delta$, in $\K[C_\Delta]$. 

     It then turns out that (see \cite[Thm.~6.2]{Kou76} or
     \cite[Prop.~2.2]{Wal99})
     \begin{displaymath}
       \dim_\K(\K[C_\Delta]/I_{\Delta})<\infty
       \;\;\;\Longleftrightarrow\;\;\;
       f\mbox{ is \ND\ along all faces of } \Delta.
     \end{displaymath}

     \lang{We should like to point out  that $f$  \ND\
     along $\Delta$ is not sufficient for the finiteness of
     $\dim_\K(\K[C_\Delta]/I_\Delta)$. Consider e.g.\
     $f=x^3+y^2$ with $\Char(\K)=3$ and 
     $\Delta=\Gamma(f)$, then $f=\In_\Delta(f)$ and $\jj(f)$ has no
     zero in the torus, yet $\K[C_\Delta]/I_\Delta=\K[x,y]/\langle
     y\rangle$ has infinite dimension. }

     The piecewise filtration induced by the $C$-polytope $P=\Gamma(f)$ determines a graded
     algebra $\gr_P(\Kx/I_P)$ for the piecewise homogeneous ideal
     $I_P=\langle x_i\cdot\In_P(f)_{x_i}\;|\;i=1,\ldots,n\rangle$.
     We can view $\K[C_\Delta]/I_\Delta$ in a natural way as
     a quotient of $\gr_P(\Kx/I_P)$, and we then get an
     injective map (see \cite[Prop.~2.6]{Kou76})
     \begin{displaymath}
       \gr_P(\Kx/I_P)\longrightarrow \bigoplus_{\Delta\text{ face of
         }\Gamma(f)} \K[C_\Delta]/I_\Delta.
     \end{displaymath}
     This shows right away that $\dim_\K(\gr_P(\Kx/I_P)$ is finite,
     if $f$ is \NND\ and convenient. 
     From this it is not hard to see that a monomial $K$-vector space basis of
     $\gr_P(\Kx/I_P)$ actually generates $M_f$ (see
     \cite[Sec.~3]{BGM10a}), and it thus follows: 
     \begin{displaymath}
       f\mbox{ is \NND\  and convenient}\;\;\;\Longrightarrow\;\;\; \mu(f)<\infty.
     \end{displaymath}
     \lang{\hfill$\Box$}
   \end{remark}

   In \cite{Kou76} Kouchnirenko \lang{not only showed that a Newton
   non-degenerate singularity $f$ is isolated, but he gives}\kurz{gave
   also} a formula
   for the Milnor number in terms of certain volumes of the faces of
   $\Gamma_-(f)$. 

   For any compact polytope $Q$ in $\R_{\geq 0}^n$ we denote by
   $V_k(Q)$ the sum of the $k$-dimensional Euclidean volumes of the
   intersections of $Q$ with the $k$-dimensional coordinate subspaces
   of $\R^n$, and following Kouchnirenko we then call
   \begin{displaymath}
     \mu_N(Q):=\sum_{k=0}^n (-1)^{n-k}\cdot k!\cdot V_k(Q)
   \end{displaymath}
   the \emph{Newton number} of $Q$. For a power series $f\in\Kx$ we
   define the \emph{Newton number} of $f$ to be
   \begin{displaymath}
     \mu_N(f):=\sup\left\{\mu_N(\Gamma_-(f_m))\;\Big|\;f_m=f+x_1^m+\ldots+x_n^m,\; m\geq 1\right\}\in\Z\cup\{\infty\}.
   \end{displaymath}
   If $f$ is convenient, then
   \begin{displaymath}
     \mu_N(f)=\mu_N(\Gamma_-(f)).
   \end{displaymath}

   The following theorem was proved by Kouchnirenko in arbitrary characteristic.

   \begin{theorem}[Kouchnirenko, \cite{Kou76}]\label{thm:kouchnirenko}
     For $f\in\Kx$ we have $\mu_N(f)\leq \mu(f)$, and if $f$ is
     \NND\ and convenient then  $$\mu(f)=\mu_N(f)<\infty.$$
   \end{theorem}

   Actually, Kouchnirenko shows that in \emph{characteristic zero} the
   equation $\mu(f)=\mu_N(f)$ still holds if $f$ is \NND\ but not convenient. We will show in Proposition
   \ref{prop:planarnnd} that at least in the planar case this also
   holds in arbitrary characteristic.

   \begin{example}\label{ex:nonnd}
     Newton non-degeneracy is sufficient but not necessary to ensure
     that the Milnor number coincides with the Newton number and both
     are finite.

     If $\Char(\K)\not=2$ then $f=(x+y)^2+xz+z^2$ is not \NND\ (see
     Remark \ref{rem:nd}), but
     \begin{displaymath}
       \mu(f)=\mu_N(f)=1.
     \end{displaymath}
     \begin{figure}[h]
       \centering
       \setlength{\unitlength}{0.4mm}
       \begin{picture}(40,50)(-10,-20)
         \linethickness{0.1mm}\scriptsize
         \thicklines\dashline[+30]{3}(0,0)(20,0)
         \thicklines\dashline[+30]{3}(0,0)(0,20)
         \thicklines\dashline[+30]{3}(0,0)(-15,-15)
         \thicklines\drawline[12](20,0)(40,0)
         \thicklines\drawline[12](0,20)(0,40)
         \thicklines\drawline[12](-15,-15)(-30,-30)
         \thicklines\drawline[12](20,0)(0,20)
         \thicklines\drawline[12](-15,-15)(20,0)
         \thicklines\drawline[12](-15,-15)(0,20)
         \put(20,0){\circle*{4}}
         \put(0,20){\circle*{4}}
         \put(-7.5,2.5){\circle*{4}}
         \put(-15,-15){\circle*{4}}
         \put(2.5,-7.5){\circle*{4}}
         
         \put(-34,-30){$x$}
         \put(40,-4){$y$}
         \put(-4,40){$z$}
       \end{picture}      
       \caption{The Newton diagram of $f=(x+y)^2+xz+z^2$}
       \label{fig:npnd}
     \end{figure}
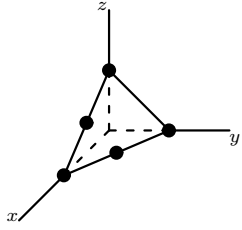
   \end{example}

   Wall proved in \cite{Wal99} the analogous result for inner
   Newton non-degenerate singularities in characteristic zero. Taking
   Theorem \ref{thm:fd=isolated} into account we generalise this to arbitrary characteristic.

   \begin{theorem}\label{thm:snnd}
     If $f\in\Kx$ is \SNND\ w.r.t.\ some $C$-polytope, then
     $$\mu(f)=\mu_N(f)=\mu_N\big(\Gamma_-(f)\big)<\infty.$$
   \end{theorem}
   \begin{proof}
     Wall introduces the $\K[C_\Delta]$-module
     \begin{displaymath}
       D_\Delta=\langle \ux^\alpha\cdot\partial_{x_i}\;|\;
       \ux^\alpha\cdot\partial_{x_i}(f)\in \K[C_\Delta]\;\;\forall\;f\in\K[C_\Delta]\rangle,
     \end{displaymath}
     generated by all
     monomial derivations which leave $\K[C_\Delta]$ invariant and considers
     the ideal
     \begin{displaymath}
       J_\Delta=\{\xi(\In_\Delta(f))\;|\;\xi\in D_\Delta\}
     \end{displaymath}
     which results by applying $D_\Delta$ to
     $\In_\Delta(f)$. He then shows that (see \cite[Prop.~2.2]{Wal99})
     \begin{displaymath}
       \dim_\K(\K[C_\Delta]/J_{\Delta})<\infty
       \;\;\;\Longleftrightarrow\;\;\;
       f\mbox{ is \SND\ along all \emph{inner} faces of }\Delta.
     \end{displaymath}
     The rings $\K[C_\Delta]/J_\Delta$ can be stacked neatly in an
     exact sequence of complexes whose homology was used by Wall to show
     (see \cite[Lem.~1.2, Prop.~2.3]{Wal99}):
     \begin{displaymath}
       f\mbox{ is \SNND }\;\;\;\Longrightarrow\;\;\;
       \mu(f)<\infty.
     \end{displaymath}
     Wall's arguments use only standard facts from toric geometry and
     homological algebra and do not depend on the characteristic of
     the base field. 

     It thus remains to show that
     $\mu(f)=\mu_N(f)=\mu_N\big(\Gamma_-(f)\big)$, but the proof for
     this is the same as in \cite[Thm.~1.6]{Wal99} if we use
     that by Theorem \ref{thm:fd=isolated} $\mu(f)<\infty$ implies that
     $f$ is finitely determined.
   \end{proof}

   \lang{
     \begin{example}[$W_{1,1}$-Singularities]
       A series $f$ with principal part
       $\In_P(f)=x^7+x^3y^2+y^4$ for $P=\Gamma(f)$ is \SNND\ if
       $\Char(\K)\not\in\{2,3,7\}$, and thus $f$ is an isolated
       singularity. 
     \end{example}
   }

   Inner Newton
   non-degeneracy has the advantage over Newton non-degeneracy that
   all right semi-quasihomogeneous singularities satisfy this condition,
   even if they are not convenient. This is an easy consequence of the
   observation in Lemma \ref{lem:singzero} as we will see in Proposition
   \ref{cor:snnd-rsqh}. 

   A polynomial $f\in\K[\ux]$ is said to be \emph{quasihomogeneous}
   \qh\  w.r.t.\ a weight vector $w\in\Z_{>0}^n$ if all monomials
   $\ux^\alpha$, $\alpha\in\supp(f)$, have the same \emph{weighted degree}
   $\deg_w(\ux^\alpha)=w\cdot\alpha=w_1\cdot\alpha_1+\ldots+w_n\cdot\alpha_n$. We call a power series
   $f=\sum_{\alpha}a_\alpha\cdot\ux^\alpha\in\Kx$ 
   \emph{right semi-quasihomogeneous} \label{page:rsqh} \rsqh\  if there is a weight
   vector $w\in\Z_{>0}^n$ such that the principal part
   $\In_w(f)=\sum\limits_{w\cdot\alpha\text{ minimal}}a_\alpha\cdot\ux^\alpha$
   has a finite Milnor number. Since in positive characteristic the
   finiteness of the Milnor number and the Tjurina number are no
   longer equivalent we have to distinguish between semi-quasihomogeneity for right
   and contact equivalence (see also \cite[Sec.~2]{BGM10a}).

   \begin{lemma}\label{lem:singzero}
     Let $f\in\K[\ux]$ be \qh\  w.r.t.\ $w\in\Z_{>0}^n$, then
     $\mu(f)<\infty$ if and only if $0$ is the only zero of $\jj(f)$.
   \end{lemma}
   \begin{proof}
     If a monomial $\ux^\alpha$ is a linear combination of
     the partial derivatives of $f$ in $\Kx$ then we only have to
     consider the suitable weighted homogeneous part, and it actually
     is a linear combination in $\K[\ux]$. Thus $\mu(f)$ is finite if
     and only if $\dim_\K(\K[\ux]/\jj(f))<\infty$. By Hilbert's
     Nullstellensatz the latter is
     equivalent to the fact that $\jj(f)$ has only finitely many zeros
     in $\K^n$. But since $f$ is weighted homogeneous for  each zero
     $q=(q_1,\ldots,q_n)$ of $\jj(f)$ also
     $(t^{w_1}q_1,\ldots,t^{w_n}q_n)$ is a zero of $\jj(f)$ for all
     $t\in\K$. Thus $\jj(f)$ has only finitely many zeros if and only
     if $0$ is the only zero of $\jj(f)$.
   \end{proof}


   \begin{proposition}\label{cor:snnd-rsqh}
     Let $P$ be a $C$-polytope with a single facet $\Delta$ with weight vector
     $w$ and suppose that $f\in\Kx$ has principal part
     $\In_w(f)=\In_\Delta(f)$ w.r.t. $P$.     
     Then $f$ is \SNND\ w.r.t.\ $P$ if and only $f$ is \rsqh\  w.r.t.\
     $w$.

     In particular, if $f$ is \rsqh\ w.r.t.\ $w$ of weighted degree
     $d$, then 
     \begin{displaymath}
       \mu(f)=\left(\frac{d}{w_1}-1\right)\cdot\ldots\cdot\left(\frac{d}{w_n}-1\right).
     \end{displaymath}
   \end{proposition}
   \begin{proof}
     Since $P$ is a $C$-polytope the unique facet $\Delta$ meets all
     coordinate subspaces except possibly $\{0\}$. Thus $f$ is \SND\
     along $\Delta$ if and only if $\Sing(\In_\Delta(f))=\{0\}$. By
     Lemma \ref{lem:singzero} this is equivalent to
     $\mu(\In_w(f))=\mu(\In_\Delta(f))<\infty$, i.e.\ that $f$ is
     \rsqh\  w.r.t.\ $w$.

     The formula for $\mu(f)$, first proved by Milnor and Orlik
     \cite{MO70} for isolated \qh\ singularities in characteristic zero,
     follows from Theorem \ref{thm:snnd} since $\mu_N(f)$ is easily
     seen to be the product of the $\frac{d}{w_i}-1$.
   \end{proof}

   \begin{example}
     Generalising Example \ref{ex:nonnd} we consider
     $f=(x+y)^k+xz^{k-1}+z^k$ for some $k\geq 2$ such that $\Char(\K)$
     neither divides $k$ nor $k-1$. Then $f$ is \qh\  w.r.t.\ $w=(1,1,1)$
     and $\Sing(f)=\{0\}$. Thus by Lemma \ref{lem:singzero} and
     Proposition \ref{cor:snnd-rsqh} $f$ is an isolated singularity and \SNND\ with
     $\mu(f)=\mu_N(f)=(k-1)^3$. Note that $f$ is not \NND.
   \end{example}

   \section{Invariants of plane curve singularities}

   In this section $f$ will always be a non-zero power series in the maximal
   ideal $\m=\langle x,y\rangle\subset\K[[x,y]]$.

   For the convenience of the reader we start this section by
   gathering numerical invariants of a
   singularity $f$ respectively numbers associated to the geometry of its
   Newton diagram that will
   be introduced and compared throughout. We will comment on these and their
   relations further down.

   \begin{remark}\label{rem:invarianten}
     Let $0\not=f\in\langle x,y\rangle\subset\K[[x,y]]$ be a power series and suppose that the Newton
     diagram of $f$ has $k$ facets $\Delta_1,\ldots,\Delta_k$.      
     By
     $l(\Delta_i)$ we denote the \emph{lattice length} of
     $\Delta_i$, i.e.\ the number of lattice points on
     $\Delta_i$ minus one. 

     We fix a minimal
     resolution of the singularity computed via successively blowing up points,
     denote by $Q\rightarrow 0$ that $Q$ is an infinitely near point of the
     origin and by $m_Q$ the multiplicity of the strict transform of $f$
     at $Q$. Finally, for
     $m\in\N$ we set $f_m:=f+x^m+y^m$.
     \medskip
     \begin{enumerate}
     \item $\mu(f):=\dim_\K(\K[[x,y]]/\langle f_x,f_y\rangle)$ is the \emph{Milnor
         number} of $f$.
     \item $\delta(f):=\sum_{Q\rightarrow 0} \frac{m_Q\cdot (m_Q-1)}{2}$ is the \emph{delta invariant}
       of $f$. 
     \item $\nu(f):=\sum_{Q\text{ special}} \frac{m_Q\cdot
         (m_Q-1)}{2}$, where an infinitely near point $Q$ is special
       if it is zero or the origin of the corresponding chart of the
       blowing up.
     \item $r(f)$ is the \emph{number of branches} of $f$ counted with
       multiplicity.
     \item If $f$ is convenient, then the \emph{Newton number} of $f$ is
       \begin{displaymath}
         \mu_N(f)=2\cdot V_2\big(\Gamma_-(f)\big)-V_1\big(\Gamma_-(f)\big)+1,
       \end{displaymath}
       and otherwise it is $\mu_N(f)=\sup\{\mu_N(f_m)\;|\;m\in\N\}$.
     \item  If $f$ is convenient, we define
       \begin{displaymath}
         \delta_N(f):=V_2(\Gamma_-(f))-\frac{V_1(\Gamma_-(f))}{2}+\frac{\sum_{i=1}^k l(\Delta_i)}{2},
       \end{displaymath}       
       and otherwise we set $\delta_N(f)=\sup\{\delta_N(f_m)\;|\; m\in\N\}$.
     \item $r_N(f):=\sum_{i=1}^k l(\Delta_i)+
       \max\{j\;|\;x^j\text{ divides }f\}+\max\{l\;|\;y^l\text{
         divides }f\}$.
     \end{enumerate}
   \end{remark}

   Coming back to the different notions of non-degeneracy, 
   a particularly interesting situation is that of plane curve
   singularities. We will end this paper by investigating this case
   more closely. One of the aims is to show that for non-degeneracy the condition
   of \emph{convenience} is often not necessary, even in positive
   characteristic. We now
   elaborate on the conditions \SND\  and \SNND\ in the planar case. 

   \begin{remark}\label{rem:planarnd}
     Let $0\not=f\in\langle x,y\rangle\subset\K[[x,y]]$ and let $P$ be a $C$-polytope such that no
     point in $\supp(f)$ lies below $P$.
     \begin{enumerate}
     \item 
       Then $f$ is \SND\ along an edge $\Delta$  of $P$ if and only if
       \begin{itemize}
       \item all zeros of $\jj(\In_\Delta(f))$ have at least one
         coordinate zero 
         if $\Delta$ does not meet any coordinate axis;
       \item all zeros of $\jj(\In_\Delta(f))$ have $x$-coordinate
         zero 
         if $\Delta$ only meets the $x$-axis;
       \item all zeros of $\jj(\In_\Delta(f))$ have $y$-coordinate
         zero 
         if $\Delta$ only meets the $y$-axis;
       \item the only zero of $\jj(\In_\Delta(f))$ is $(0,0)$ 
         if $\Delta$ meets both axes.
       \end{itemize}
     \item In \cite{Wal99} Wall describes how much the $C$-polytope $P$
       may differ from $\Gamma(f)$ if $f$ is \SNND\ w.r.t.\ $P$:
       \begin{itemize}
       \item each inner vertex of $P$ is a vertex of $\Gamma(f)$;
       \item an edge of $P$ which does not meet a coordinate axis is
         an edge of $\Gamma(f)$;
       \item an edge of $P$ which meets exactly one of the coordinate
         axes is either itself an edge of $\Gamma(f)$ or, replacing the
         point on the coordinate axis by the point on the edge with
         distance one from the coordinate axis, leads to 
         an edge or a vertex of $\Gamma(f)$;
       \item if $P$ consists of a single edge meeting both coordinate axes, then $f$ is
         \rsqh\  w.r.t.\ any weight vector defining this edge (see Prop.\
         \ref{cor:snnd-rsqh}); in particular, the principal part of $f$
         is reduced and unless it is $xy$ the edge $P$ contains an
         edge of the Newton diagram whose end points have distance at
         most one from the corresponding axes.
       \end{itemize}
       Wall gives this characterisation over the complex numbers, but
       it actually holds in the same way in any characteristic.
     \end{enumerate}
   \end{remark}

   It turns out that in the planar situation \NND\ implies \SNND.

   \begin{proposition}\label{prop:nnd-snnd}
     If $0\not=f\in\langle x,y\rangle\subset\K[[x,y]]$ is \NND\ then $f$ is \SNND\ w.r.t.\
     $\Gamma(f)$, i.e.\ $f$ is \SND\ along each inner face of $\Gamma(f)$.
   \end{proposition}
   \begin{proof}
     Let $\Delta$ be any inner face of $\Gamma(f)$. If
     $\Delta$  intersects none of the two coordinate axes, then $f$
     is \SND\ along $\Delta$ since it is \ND\ along $\Delta$ by Remark
     \ref{rem:nd}. 

     If $\Delta$ meets the $y$-axis we have to show that there is no
     zero of $\jj(\In_\Delta(f))$ with non-zero $y$-coordinate. 
     In this situation $\Delta$ is an edge of the Newton diagram
     whose one end point lies on the $y$-axis, i.e.\
     \begin{displaymath}
       \In_\Delta(f)=a\cdot y^k+x\cdot g
     \end{displaymath}
     for some $a\in\K^*$, $k\geq 1$ and $g\in\K[x,y]$. 

     By assumption $\jj(\In_\Delta(f))$ has no zero in $(\K^*)^2$ and
     we have to exclude the possibility that it has a zero
     $q=(0,z)\in\{0\}\times\K^*$. This is the case since
     \begin{displaymath}
       \In_\Delta(f)_y=a\cdot k\cdot y^{k-1}+x\cdot g_y
     \end{displaymath}
     and 
     \begin{displaymath}
       \In_\Delta(f)_y(q)=a\cdot k\cdot z^{k-1}\not=0,
     \end{displaymath}
     where, for the second statement, we note that by Remark \ref{rem:nd}
     $\Char(\K)$ does not divide $k$. 
     Similarly, if $\Delta$ meets the $x$-axis there is no zero of
     $\jj(\In_\Delta(f))$ with non-zero $x$-coordinate.

     Thus $f$ is also \SND\ along any inner face which meets any of
     the two coordinate axes by Remark \ref{rem:planarnd}, and
     altogether we have that $f$ is \SND\ 
     w.r.t.     $\Gamma(f)$. 
   \end{proof}

   Since on each face \SND\ implies \ND, the previous result can be
   rephrased as follows.

   \begin{corollary}
     If $0\not=f\in\langle x,y\rangle\subset\K[[x,y]]$, then the following are equivalent:
     \begin{enumerate}
     \item $f$ is \NND.
     \item $f$ is \SNND\ w.r.t.\ $\Gamma(f)$, and in case $\Gamma(f)$ meets
       the $x$-axis or the $y$-axis then the corresponding coordinate is
       not divisible by $\Char(\K)$.
     \end{enumerate}
   \end{corollary}

   We show now that in the planar case Kouchnirenko's result holds in arbitrary
   characteristic without the assumption that $f$ is convenient.

   \begin{proposition}\label{prop:planarnnd}
     Suppose that $0\not=f\in\langle x,y\rangle\subset\K[[x,y]]$ is \NND, then $\mu(f)=\mu_N(f)$.
   \end{proposition}
   \begin{proof}
     We may assume that $f\in\m^2$. Moreover,
     if $\Gamma(f)$ consists of a single point $\alpha$ then either
     $\alpha=(1,1)$ with $\mu(f)=\mu_N(f)=1$ or
     $\mu(f)=\mu_N(f)=\infty$. We thus also may assume that
     $\Gamma(f)$ has at least one edge.
     
     Let $\alpha=(k,l)$ be the end point of $\Gamma(f)$ closest to the
     $y$-axis, and suppose that $k\geq 2$, then $\mu_N(f)=\infty$ and
     by Theorem \ref{thm:kouchnirenko} also $\mu(f)=\infty$. Thus we
     may assume that either $\alpha$ is on the $y$-axis or its distance $k$ to the
     $y$-axis is one. Similarly, we may assume that the end point of
     $\Gamma(f)$ closest to the $x$-axis has distance at most one from
     the $x$-axis.

     Note that there is a unique $C$-polytope $P$ which contains $\Gamma(f)$ and
     which has the same number of edges. It is derived from
     $\Gamma(f)$ by prolonging the obvious edges to the coordinate
     axes. We want to show that $f$ is \SNND\ w.r.t.\ $P$. 

     Let $\Delta$ be an inner face of $P$. If $\Delta$ does not meet
     any of the coordinate axes, then $\Delta$ is a face of
     $\Gamma(f)$ and condition \ND\ implies that $f$ is also \SND\
     along $\Delta$. 

     If $\Delta$ meets the $y$-axis, then it prolongs
     the edge $\Delta'$ of $\Gamma(f)$ whose end point closest to the $y$-axis
     is $\alpha=(k,l)$ with $k\leq 1$. If $k=0$ then $\Delta=\Delta'$ is an
     edge of $\Gamma(f)$ and we can see as in the proof of Proposition
     \ref{prop:nnd-snnd} that $\jj(\In_\Delta(f))$ has no zero 
     with non-zero $y$-coordinate. If $k=1$ then
     \begin{displaymath}
       \In_\Delta(f)=\In_{\Delta'}(f)=a\cdot x\cdot y^l+x^2\cdot g
     \end{displaymath}
     for some $a\in\K^*$ and some $g\in \K[x,y]$. Since
     $f$ satisfies \ND\ along $\Delta'$ there is no point 
     $q\in\Sing(\In_\Delta(f))$ with both coordinates non-zero, and
     since 
     \begin{displaymath}
       \In_\Delta(f)_x=a\cdot y^l+x\cdot (2g+x\cdot g_x)
     \end{displaymath}
     there can also be no point $q\in\Sing(\In_\Delta(f))$ with only
     the $y$-coordinate non-zero. Thus, in any case we see that
     $\jj(\In_\Delta(f))$ has no zero with a non-zero
     $y$-coordinate. 
     Similarly, if $\Delta$ meets the $x$-axis there is no zero of $\jj(\In_\Delta(f))$
     with a non-zero $x$-coordinate. 

     Again by Remark \ref{rem:planarnd} $f$ is \SND\ along each inner
     face of $P$ which meets any of the coordinate axes, and thus
     altogether $f$ is \SNND. Theorem \ref{thm:snnd} implies
     that $\mu(f)=\mu_N(f)$.
   \end{proof}

   \begin{example}
     We now give an example for a reduced power series which is not
     \SNND\ w.r.t.\ any $C$-polytope.
     Let $f=x^6+y^3+x^5y\in\K[[x,y]]$
     with $\Char(\K)=2$ and suppose that $f$ is \SNND\ w.r.t.\ some
     $C$-polytope $P$. By Remark~\ref{rem:planarnd} $P$ must be the
     Newton diagram of $f$ and by Proposition \ref{cor:snnd-rsqh} $f$ is then
     \rsqh\  in contradiction to $\mu(\In_P(f))=\mu(x^6+y^3)=\infty$. Note
     that $\mu(f)=13>10=\mu_N(f)$.
     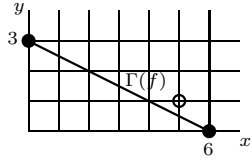
\begin{figure}[h]
       \centering
       \setlength{\unitlength}{0.4mm}
       \begin{picture}(70,50)(0,0)
         \multiput(0,0)(10,0){7}{\line(0,1){40}}
         \multiput(0,0)(0,10){4}{\line(1,0){70}}
         \linethickness{0.1mm}\scriptsize
         \thicklines\drawline[12](60,0)(0,30)
         \put(60,0){\circle*{4}}
         \put(0,30){\circle*{4}}
         \put(50,10){\circle{4}}
         \put(58,-8){$6$}
         \put(-7,29){$3$}
         \put(70,-5){$x$}
         \put(-5,40){$y$}
         \put(32,15){$\Gamma(f)$}
       \end{picture}       
       \caption{The Newton diagram of $f=x^6+y^3+x^5y$.}
       \label{fig:snnd}
     \end{figure}
   \end{example}

   We are now going to prove that Milnor's formula
   $\mu(f)=2\cdot\delta(f)-r(f)+1$ holds if $f$ is \NND.
   Beelen and Pellikaan investigate in \cite{BP00} plane curve
   singularities in arbitrary characteristic, and under the assumption
   of convenience and weak Newton non-degeneracy they give a formula
   for \lang{the delta   invariant} $\delta(f)$ \lang{of $f$} in terms of the Newton diagram. We generalise
   this by dropping the condition of convenience. Moreover, we show
   that if $f$ is \WND\ along an edge $\Delta$ of
   $\Gamma(f)$ of lattice length $k$, then there are exactly $k$
   branches of $f$ corresponding to $\Delta$. Combining these  results Milnor's
   formula with the Newton number instead of the Milnor number follows
   in arbitrary characteristic.
   
   \medskip

   \lang{If $f\in\K[[x,y]]$ is convenient and $\Delta_1,\ldots,\Delta_k$ are
     the facets of the Newton diagram $\Gamma(f)$, then we define
     \begin{displaymath}
       \delta_N(f):=V_2(\Gamma_-(f))-\frac{V_1(\Gamma_-(f))}{2}+\frac{\sum_{i=1}^k l(\Delta_i)}{2},
     \end{displaymath}
     where $l(\Delta_i)$ is the lattice length of $\Delta_i$, i.e.\ one
     less than the number of lattice points on $\Delta_i$. If $f$ is not
     convenient, we generalise this definition to
     \begin{displaymath}
       \delta_N(f):=\sup\{\delta_N(f_m)\;|\; f_m=f+x^m+y^m, m\in\N\}.
     \end{displaymath}
   }
   \kurz{In the following we study the numbers $\delta_N$, $r_N$ and
     $\mu_N$ which depend only on the Newton diagram and compare them
     with the singularity invariants $\nu$, $\delta$, $r$ and $\mu$
     (see Remark \ref{rem:invarianten}).}
   
   \begin{example}\label{ex:nuN}
     If $f=x^4y+x^2y^2+y^5$ and $m\geq 6$, then the Newton diagram of
     $f_m$ has three facets $\Delta_1,\Delta_2,\Delta_3$ of lattice
     length one (see Figure \ref{fig:nu}).
     We thus get
     \begin{align*}
       \delta_N(f_m)=&\;V_2(\Gamma_-(f_m))-\frac{V_1(\Gamma_-(f_m))}{2}+\frac{l(\Delta_1)+l(\Delta_2)+l(\Delta_3)}{2}\\=&
       \;\left(10+\frac{m-4}{2}\right)-\frac{5+4+(m-4)}{2}+\frac{1+1+1}{2}=7,
     \end{align*}
     where the $\frac{m-4}{2}$ corresponds to both, the area of the
     gray triangle in Figure \ref{fig:nu} and half the length of the
     intersection of this triangle with the $x$-axis. We thus have 
     \begin{displaymath}
       \delta_N(f)=\delta_N(f_6)=7.
     \end{displaymath}
     \begin{figure}[h]
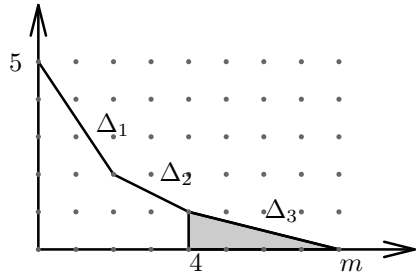

       \centering
       \begin{texdraw}
         \arrowheadtype t:V
         \drawdim cm \relunitscale 0.5 \linewd 0.06 
         \move (4 0) \lvec (4 1) \lvec (8 0) \lvec (4 0) \lfill f:0.8
         \move (0 0) \avec (10 0)
         \move (0 0) \avec (0 6.5)
         \move (0 5) \lvec (2 2) \lvec (4 1) \lvec (8 0)
         \setgray 0
         \htext (-0.8 4.7){$5$}
         \htext (4 -0.6){$4$}
         \htext (8 -0.6){$m$}
         \htext (1.5 3){$\Delta_1$}         
         \htext (3.2 1.7){$\Delta_2$}         
         \htext (6 0.7){$\Delta_3$}         
         \move (0 0) \fcir f:0.4 r:0.06 
         \move (1 0) \fcir f:0.4 r:0.06 
         \move (2 0) \fcir f:0.4 r:0.06 
         \move (3 0) \fcir f:0.4 r:0.06 
         \move (4 0) \fcir f:0.4 r:0.06 
         \move (5 0) \fcir f:0.4 r:0.06 
         \move (6 0) \fcir f:0.4 r:0.06 
         \move (0 1) \fcir f:0.4 r:0.06 
         \move (1 1) \fcir f:0.4 r:0.06 
         \move (2 1) \fcir f:0.4 r:0.06 
         \move (3 1) \fcir f:0.4 r:0.06 
         \move (4 1) \fcir f:0.4 r:0.06 
         \move (5 1) \fcir f:0.4 r:0.06 
         \move (6 1) \fcir f:0.4 r:0.06 
         \move (0 2) \fcir f:0.4 r:0.06 
         \move (1 2) \fcir f:0.4 r:0.06 
         \move (2 2) \fcir f:0.4 r:0.06 
         \move (3 2) \fcir f:0.4 r:0.06 
         \move (4 2) \fcir f:0.4 r:0.06 
         \move (5 2) \fcir f:0.4 r:0.06 
         \move (6 2) \fcir f:0.4 r:0.06 
         \move (0 3) \fcir f:0.4 r:0.06 
         \move (1 3) \fcir f:0.4 r:0.06 
         \move (2 3) \fcir f:0.4 r:0.06 
         \move (3 3) \fcir f:0.4 r:0.06 
         \move (4 3) \fcir f:0.4 r:0.06 
         \move (5 3) \fcir f:0.4 r:0.06 
         \move (6 3) \fcir f:0.4 r:0.06 
         \move (0 4) \fcir f:0.4 r:0.06 
         \move (1 4) \fcir f:0.4 r:0.06 
         \move (2 4) \fcir f:0.4 r:0.06 
         \move (3 4) \fcir f:0.4 r:0.06 
         \move (4 4) \fcir f:0.4 r:0.06 
         \move (5 4) \fcir f:0.4 r:0.06 
         \move (6 4) \fcir f:0.4 r:0.06 
         \move (0 5) \fcir f:0.4 r:0.06 
         \move (1 5) \fcir f:0.4 r:0.06 
         \move (2 5) \fcir f:0.4 r:0.06 
         \move (3 5) \fcir f:0.4 r:0.06 
         \move (4 5) \fcir f:0.4 r:0.06 
         \move (5 5) \fcir f:0.4 r:0.06 
         \move (6 5) \fcir f:0.4 r:0.06 
         \move (7 0) \fcir f:0.4 r:0.06 
         \move (7 1) \fcir f:0.4 r:0.06 
         \move (7 2) \fcir f:0.4 r:0.06 
         \move (7 3) \fcir f:0.4 r:0.06 
         \move (7 4) \fcir f:0.4 r:0.06 
         \move (7 5) \fcir f:0.4 r:0.06 
         \move (8 0) \fcir f:0.4 r:0.06 
         \move (8 1) \fcir f:0.4 r:0.06 
         \move (8 2) \fcir f:0.4 r:0.06 
         \move (8 3) \fcir f:0.4 r:0.06 
         \move (8 4) \fcir f:0.4 r:0.06 
         \move (8 5) \fcir f:0.4 r:0.06 
       \end{texdraw}
       \caption{The Newton diagram of $x^6+x^2y^2+y^5$.}
       \label{fig:nu}
     \end{figure}     
   \end{example}

   The number $\delta_N(f)$ is related to the delta invariant of $f$. If
   we consider a minimal resolution of the singularity computed via
   successive blowing up and denote by
   $Q\rightarrow 0$ that $Q$ is an infinitely near point of the
   origin, then 
   \lang{we know that (see \cite[Prop.~3.34]{GLS07})
     \begin{displaymath}
       \delta(f)=\sum_{Q\rightarrow 0} \frac{m_Q\cdot (m_Q-1)}{2},
     \end{displaymath}
     where $m_Q$ denotes the multiplicity of the strict transform of $f$
     at $Q$. Beelen and Pellikaan introduced the number
     \begin{displaymath}
       \nu(f):=\sum_{Q\text{ special}}\frac{m_Q\cdot(m_Q-1)}{2}\leq \delta(f),
     \end{displaymath}
     where an infinitely near point is \emph{special} if it is $0$ or the
     origin in the corresponding chart of the blowing up procedure. }
   \kurz{the inequality $\nu(f)\leq\delta(f)$ follows from the definition of $\nu$
     and $\delta$.}
   Clearly, $\nu(f)$ depends on the coordinates of $f$, while
   $\delta(f)$ does not. 
   Beelen and Pellikaan show that $\nu(f)$
   and $\delta_N(f)$ coincide if $f$ is convenient. Using our
   generalisation of $\delta_N$, the condition of
   convenience can be dropped.

   \begin{lemma}\label{lem:nuN}
     If $0\not=f\in\langle x,y\rangle\subset\K[[x,y]]$,
     then $\nu(f)=\delta_N(f)$.
   \end{lemma}
   \begin{proof}
     If $x^2$ or $y^2$ divides $f$ then both numbers are infinite, so
     we may assume that this is not the case.

     If $y$ divides $f$ then passing from $f$ to $f+x^m$ for some
     large $m$ replaces the smooth branch $y$ by some other smooth
     branch with the same tangent direction, and the analogous
     argument holds if $x$ divides $f$. Therefore, $\nu(f)=\nu(f_m)$
     for sufficiently large $m$. Moreover, as in Example \ref{ex:nuN}
     the values of $\delta_N(f_m)$ stabilise for sufficiently large $m$,
     since the area that is added in the computation of
     $V_2(\Gamma_-(f_m))$ coincides with the length that is subtracted
     in the computation of $V_1(\Gamma_-(f_m))$. Using 
     \cite[Thm.~3.11]{BP00} we can summarise that
     for a sufficiently large $m$
     \begin{displaymath}
       \delta_N(f)=\delta_N(f_m)=\nu(f_m)=\nu(f).
     \end{displaymath}
   \end{proof}

   One would like to know under which conditions $\nu(f)$ actually coincides
   with $\delta(f)$, and Beelen and Pellikaan show in
   \cite[Prop.~3.17]{BP00} that for a convenient $f$ weak Newton
   non-degeneracy is a sufficient condition to assure this. Again
   we can  drop the condition of convenience. 

   \begin{proposition}\label{prop:nudelta}
     For $0\not=f\in\langle x,y\rangle$ we have $\delta_N(f)\leq\delta(f)$, and if $f$ is \WNND\ then $\delta_N(f)=\delta(f)$.
   \end{proposition}
   \begin{proof}
     If $f$ is divisible by $x^2$ or $y^2$, then all of these numbers
     are infinite, so we may exclude this case. Moreover,
     we may restrict to the case that $y$ divides $f$ but $x$ does not, as
     the remaining cases work analogously.
     As above, passing from $f$ to $f+x^m$ for
     a large $m$ replaces the smooth branch $y$ by some smooth branch
     with the same tangent direction, so the delta invariant does not
     change. Moreover, if $m$ is sufficiently large then $\Gamma(f)$
     differs from $\Gamma(f_m)$ by one additional facet $\Delta$, a line
     segment with end points $(m,0)$ and $(k,1)$. The initial form
     along $\Delta$ is $\In_\Delta(f_m)=x^m+c\cdot x^ky$ and
     $\jj(x^m+c\cdot x^ky)$ has no zero in the torus
     $(\K^*)^2$. Therefore, $f_m$ is convenient and \WNND, so that
     \cite[Prop.~3.17]{BP00} and Lemma~\ref{lem:nuN} imply 
     for sufficiently large $m$  
     \begin{displaymath}
       \delta(f)=\delta(f_m)=\nu(f_m)=\delta_N(f_m)=\delta_N(f)=\nu(f).
     \end{displaymath}
   \end{proof}

   \lang{We denote by $r(f)$ the number of branches of $f$ counted with
     multiplicity, i.e.\ the number of irreducible factors of
     $f$. Moreover, we introduce the combinatorial counterpart of $r$ as
     \begin{displaymath}
       r_N(f)=\sum_{i=1}^k l(\Delta_i)+
       \max\{j\;|\;x^j\text{ divides }f\}+\max\{l\;|\;y^l\text{ divides }f\}.
     \end{displaymath}}
   \kurz{Now we compare the number of branches $r(f)$ of $f$ with its
     combinatorial counterpart $r_N(f)$.}
   The following result is implicit in Beelen and Pellikan (\cite{BP00}).

   \begin{lemma}\label{lem:r_N}
     For $0\not=f\in\langle x,y\rangle$ we have $r(f)\leq r_N(f)$, and if $f$ is \WNND\ then $r_N(f)=r(f)$.
   \end{lemma}
   \begin{proof}
     If $j$ and $l$ are the maximal such that $x^j$ and $y^l$ divide
     $f$, then
     \begin{displaymath}
       r_N(f)=\sum_{i=1}^k l(\Delta_i)+j+l.
     \end{displaymath}
     It is well known that the 
     lattice length of a facet of the Newton diagram of $f$ is an
     upper bound for
     the number of branches of $f$ corresponding to this facet. This
     implies the inequality $r(f)\leq r_N(f)$.
     The proof of \cite[Prop.~3.17]{BP00} shows then that $f$ has
     indeed $l(\Delta_i)$ branches corresponding to $\Delta_i$, if $f$
     is \WND\ along $\Delta_i$ (see also
     \cite[Prop.~3.18]{BP00}). This shows that  $f$ has exactly
     $r_N(f)$ branches, counting the branches $x$ and $y$ with
     multiplicity, if $f$ is \WNND. 
   \end{proof}


   \begin{lemma}\label{lem:milnormuN}
     If $0\not=f\in\langle x,y\rangle\subset\K[[x,y]]$, then $\mu_N(f)=2\cdot\delta_N(f)-r_N(f)+1$.
   \end{lemma}
   \begin{proof}
     If $x^2$ or $y^2$ divides $f$, then both sides of the equation
     are infinite, and we may thus assume that this is not the case.

     Suppose now that $\Gamma(f)$ has the facets
     $\Delta_1,\ldots,\Delta_k$ and let $m$ be very
     large. Then $\Gamma(f_m)$ also has the facets
     $\Delta_1,\ldots,\Delta_k$, and it has an additional facet of
     lattice length one if $x$ divides $f$ and the same for $y$.
     In particular,  $r_N(f)=r_N(f_m)$.

     Since $f_m$ is convenient the
     definition of $\mu_N$, $\delta_N$ and $r_N$ gives right away
     \begin{displaymath}
       \mu_N(f_m)=2\cdot \delta_N(f_m)-r_N(f_m)+1.
     \end{displaymath}
     Moreover, for sufficiently large $m$ we have
     $\mu_N(f_m)=\mu_N(f)$ and $\delta_N(f_m)=\delta_N(f)$, and hence
     \bmath
       \mu_N(f)=2\cdot \delta_N(f)-r_N(f)+1.
     \emath
   \end{proof}

   Combining the last three results we get the following
   generalisation of the result of Beelen and Pellikan.

   \begin{theorem}\label{thm:beelen-pellikaan}
     If $0\not=f\in\langle x,y\rangle\subset\K[[x,y]]$ is \WNND, then
     $\mu_N(f)=2\cdot\delta(f)-r(f)+1$. 
   \end{theorem}
   \begin{proof}
     The result follows from Lemma \ref{lem:milnormuN}, 
     Proposition \ref{prop:nudelta} and Lemma \ref{lem:r_N}.
   \end{proof}

   Together with Kouchnirenko's formula for the Milnor number in
   Proposition \ref{prop:planarnnd} we deduce then that Milnor's formula
   $\mu(f)=2\cdot\delta(f)-r(f)+1$ in characteristic zero (see
   \cite{Mil68} or \cite[Prop.~3.35]{GLS07})
   holds in arbitrary characteristic for Newton non-degenerate
   singularities, even without the condition of convenience. 

   \begin{theorem}\label{thm:milnorformula}
     If $0\not=f\in\langle x,y\rangle\subset\K[[x,y]]$ is \NND, then
     $\mu(f)=2\cdot\delta(f)-r(f)+1$. 
   \end{theorem}
   
   Without the assumption of Newton non-degeneracy one has at least an
   inequality as was proved by Melle and Wall
   \cite[Formula~(14)]{MW01} based on a result by Deligne
   \cite[Theorem~2.4]{Del73}. We are grateful to Alejandro Melle for
   pointing out this result to us.

   \begin{proposition}[Deligne, Melle-Wall]\label{prop:mellewall}
     If $0\not=f\in\langle x,y\rangle\subset\K[[x,y]]$, then $\mu(f)\geq 2\cdot\delta(f)-r(f)+1$.
   \end{proposition}
    
   The difference of the two sides is measured by the so called Swan
   character which counts  \emph{wild} vanishing cycles that can
   only occur in positive characteristic. For details we refer to
   \cite{MW01} and \cite{Del73}. The last two results imply

   \begin{corollary}\label{cor:mellewall}
   There are no wild vanishing cycles in positive characteristic if f is NND.
   \end{corollary}

   Note that we always have the inequalities
   \begin{displaymath}
     \mu_N(f)\leq 2\cdot \delta_N(f)-r(f)+1\leq 2\cdot \delta(f)-r(f)+1\leq\mu(f).
   \end{displaymath}
   It is
   easy to see that the equality may be violated in positive
   characteristic, and that the above inequalities may be
   strict. E.g.\ $\Char(\K)=2$, $f=(x-y)^2+x^5$ then $\mu_N(f)=1$, $\delta_N(f)=1$,
   $\delta(f)=2$, $r(f)=1$, $\mu(f)=\infty$, so that
   \begin{displaymath}
     \mu_N(f)<2\cdot\delta_N(f)-r(f)+1<2\cdot \delta(f)-r(f)+1<\mu(f).
   \end{displaymath}
   Note that the first two inequalities hold in characteristic zero as
   well.

   We can now use the above results to measure the difference between
   $\mu(f)$ and $\mu_N(f)$ better and thereby generalise a result of
   P{\l}oski \cite{Plo99}, who proved this for $\K=\C$ and $f$ convenient.

   \begin{proposition}
     If $0\not=f\in\langle x,y\rangle\subset\K[[x,y]]$, then $\mu(f)-\mu_N(f)\geq r_N(f)-r(f)\geq 0$.
   \end{proposition}
   \begin{proof}
     Combining Proposition \ref{prop:mellewall} with Lemma
     \ref{lem:nuN}, Lemma \ref{lem:r_N} and Lemma
     \ref{lem:milnormuN} we get
     \begin{align*}
       \mu(f)\;\geq\; &2\cdot\delta(f)-r(f)+1 \geq 
       2\cdot\delta_N(f)-r(f)+1\\\;=\;&2\cdot
       \delta_N(f)-r_N(f)+1+(r_N(f)-r(f))\\
       \;=\;&\mu_N(f)+(r_N(f)-r(f))\geq \mu_N(f),
     \end{align*}
     which proves the claim.
   \end{proof}



\providecommand{\bysame}{\leavevmode\hbox to3em{\hrulefill}\thinspace}

\end{document}